\documentclass[12pt,reqno]{amsart}

\usepackage{a4wide}
\usepackage{amsmath}
\usepackage[english]{babel}
\usepackage{amsfonts}
\usepackage{amssymb}
\usepackage{dsfont}
\usepackage{bm}
\usepackage{mathrsfs}
\usepackage{graphicx}
\usepackage{fancyhdr}
\usepackage{amsthm}
\usepackage{empheq}
\usepackage{cases}
\usepackage[all]{xy}
\usepackage{stmaryrd}
\usepackage[colorlinks, citecolor=blue, linkcolor=red]{hyperref}

\usepackage{mathtools}
\mathtoolsset{showonlyrefs,showmanualtags}

\def\RR{{\mathbb R}}

\newcommand{\eps}{\varepsilon}

\newtheorem{ackn}{Acknowledgments\!}

\newtheorem{theorem}{{Theorem}}[section]
\newtheorem{lemma}[theorem]{{Lemma}}

\newtheorem{corollary}[theorem]{{Corollary}}

\theoremstyle{definition}

\numberwithin{equation}{section}

\begin{document}

\title[Semilinear elliptic equations on manifolds with nonnegative Ricci curvature]{Semilinear elliptic equations \\on manifolds with nonnegative Ricci curvature
}

\author{Giovanni Catino}
\address{Giovanni Catino\\
	Dipartimento di Matematica,
	Politecnico di Milano, Piazza Leonardo da Vinci 32, 20133, Milano,
	Italy}
\email{giovanni.catino@polimi.it}

\author[Dario D. Monticelli]{Dario D. Monticelli}
\address[Dario D. Monticelli]{Dipartimento di Matematica, Politecnico di Milano, Piazza Leonardo da Vinci 32, 20133 Milano, Italy}
\email[]{dario.monticelli@polimi.it}


\begin{abstract}
In this paper we prove classification results for solutions to subcritical and critical semilinear elliptic equations with a nonnegative potential on noncompact manifolds with nonnegative Ricci curvature. We show in the subcritical case that all nonnegative solutions vanish identically. Moreover, under some natural assumptions, in the critical case we prove a strong rigidity result, namely we classify all nontrivial solutions showing that they exist only if the potential is constant and the manifold is isometric to the Euclidean space.
\end{abstract}


\maketitle


\section{Introduction}

Let $(M^n, g)$, $n\geq 2$, be a smooth complete (with no boundary), noncompact, $n$-dimensional Riemannian manifold with nonnegative Ricci curvature. In this paper we consider nonnegative solutions to the semilinear elliptic equation
\begin{equation}\label{eq-eq}
-\Delta_g u = K u^p
\end{equation}
where $\Delta_g$ is the Laplace-Beltrami operator and $K$ is a smooth nonnegative function on $M$. We restrict our attention to the superlinear and subcritical case
$$
1<p< p_c\qquad\text{where }\qquad p_c=\begin{cases}+\infty &\text{if }n=2\\\frac{n+2}{n-2} &\text{if }n\geq 3\end{cases}
$$
and to the critical case, $p=p_c=\frac{n+2}{n-2}$ when $n\geq 3$. In case $n=2$, we deal with solutions to the critical equation with exponential nonlinearity
\begin{equation}\label{eq-eq2}
-\Delta_g u = K e^u.
\end{equation}
If we denote by $R=R_g$ the scalar curvature of the metric $g$, it is known that these critical equations arise in the problem of prescribing the scalar curvature of a conformal metric when the original metric has zero scalar curvature. More precisely, if $n\geq 3$, then the scalar curvature of the metric $\tilde{g}=u^{\frac{4}{n-2}}g$, $u>0$, is given by
$$
-\Delta_g u +\tfrac{n-2}{4(n-1)} R_g= R_{\tilde{g}} \,u^{\frac{n+2}{n-2}}
$$
while in dimension $n=2$, the corresponding equation for the conformal change $\tilde{g}=e^u g$ reads as
$$
-\Delta_g u +R_g=R_{\tilde{g}}\,  e^{u}.
$$
It is natural to expect stronger classification results when $K\geq 0$ which will be the case we study in this paper. In particular, we will first state our results in the simpler but geometrically relevant case $K\equiv 1$. This choice of $K$ corresponds to the so called Yamabe problem, when $(M^n,g)$ is scalar flat and hence Ricci flat. It is clear that solutions to \eqref{eq-eq} or \eqref{eq-eq2} are trivial when $M$ is compact.

\medskip

In the Euclidean setting, problem \eqref{eq-eq} with $K\equiv 1$ is now well understood. Gidas and Spruck \cite{gidspr} showed that the only nonnegative solution when $1<p<p_c$ is $u\equiv 0$ on $\RR^n$, via test functions argument. Indeed, this is a consequence of the following general result that they proved in case $n\geq 3$.
\begin{theorem}[\cite{gidspr}]\label{t-subk1} Let $(M^n,g)$ be a  complete noncompact Riemannian manifold with nonnegative Ricci curvature and let $u\in C^{2}(M)$ be a nonnegative solution of
$$
-\Delta u = u^p\quad\text{in }M,\qquad\text{with } 1<p<p_c.
$$
Then $$u\equiv 0\quad\text{on }M.$$
\end{theorem}
It is known that the same result holds also when $n=2$. In Theorem \ref{t-sub} we provide a simpler proof of a more general result which includes Theorem \ref{t-subk1} for every $n\geq 2$.

\medskip

The explicit positive solutions to the critical equation
\begin{equation}\label{eq-crit127}
-\Delta u = u^{\frac{n+2}{n-2}}\quad\text{on }\RR^n,
\end{equation}
with $n\geq 3$, are given by
\begin{equation}\label{eq-rad}
u(x)=\left(\frac{1}{a + b|x-x_0|^2}\right)^{\frac{n-2}{2}},
\end{equation}
with $a,b>0$, $1=n(n-2)ab$ and $x_0\in\RR^n$. These functions were constructed by Aubin \cite{aub} and Talenti \cite{tal} as minimizers of
$$
S_g(M):=\inf_{0\not\equiv u\in D^{1,2}(M)}\frac{\int_M |\nabla u|^2\,dV_g}{\left(\int_M u^{\frac{2n}{n-2}}\,dV_g\right)^{\frac{n-2}{n}}}
$$
where $dV_g$ is the canonical volume element and
$$
D^{1,2}(M)=\left\{u\in L^{\frac{2n}{n-2}}(M): |\nabla u|\in L^2(M)\right\},
$$
with $M=\RR^n$. We note that, if $S_g(M)>0$, then it is the best constant in the Sobolev embedding. Caffarelli, Gidas and Spruck \cite{cafgidspr} (see also \cite{cheli, lizha}) proved that any solution to \eqref{eq-crit127} is radial and hence given by \eqref{eq-rad}.  Essential tools in their proof are the moving planes method and the Kelvin transform that allow to reduce the problem to the study of (singular) solutions that have nice decaying properties at infinity. Previous results were proved by Gidas, Ni and Nirenberg \cite{gidninir} and Obata \cite{oba} under the additional assumption that $u$ decays as $|x|^{-(n-2)}$ at infinity. In case $n=2$ solutions to
$$\begin{cases}-\Delta u= e^u\quad\text{in }\RR^2 \\ \int_{\RR^2}  e^u <+\infty \end{cases} $$
were classified by Chen and Li \cite{cheli}, who showed that
$$
u(x)=\log\frac{1}{(a+b|x-x_0|^2)^2},
$$
for some $a,b>0$ with $1=8ab$ and $x_0\in\RR^2$. Their method relies on the moving plane method and on a previous result by Brezis and Merle \cite{bremer} on the upper boundedness of solutions.

\medskip

Extending such classification results for the critical equations to the case of more general Riemannian manifolds requires the introduction of different techniques than those used in the Euclidean setting, which strongly rely on the conformal invariance of the problem and the rich structure of the conformal group of the ambient space. In this paper we study the problem in the natural setting of a complete, noncompact Riemannian manifold $(M,g)$ with nonnegative Ricci curvature, also allowing for the presence of a nonnegative potential function $K$.

\medskip

Here and in the rest of the paper we will denote by $Ric=Ric_g$, $R=R_g$, $dV_g$ and $r(\cdot)$ the Ricci curvature, the scalar curvature, the canonical Riemannian volume form and the geodesic distance from a fixed reference point of $M$, respectively.

\medskip

The novelty of our approach consists in using a careful test functions argument which starts from the classical Bochner formula
$$
\frac12 \Delta|\nabla f|^2 = |\nabla^2 f|^2 + \operatorname{Ric}(\nabla f,\nabla f)+\langle \nabla f, \nabla \Delta f\rangle,
$$
through which we are able to prove integral estimates involving the squared norm of the traceless Hessian
$$
\mathring{\nabla}^2 f = \nabla^2 f - \frac{\Delta f}{n}\,g
$$
of a suitable power of the solution and the Ricci tensor in the direction of the gradient of the solution. Under very general assumptions, we can then show that both quantities must vanish identically on $M$ and this leads to the classification of nontrivial solutions and the rigidity of the ambient manifold, by using a characterization of conformal gradient vector fields. The starting point of our approach is partly reminiscent of the method used by Gidas and Spruck \cite{gidspr} for subcritical equations ($n\geq 3$), Bidaut-V\'eron and Raoux \cite{bidrao} for  subcritical systems and the first author with Castorina and Mantegazza \cite{cascatman} for subcritical parabolic equations.

\medskip

In the first theorem we deal with the case $n\geq 3$ and finite energy solutions, i.e. $u\in D^{1,2}(M)$.

\begin{theorem}\label{t-crik1} Let $(M^n,g)$, $n\geq 3$, be a complete noncompact Riemannian manifold with nonnegative Ricci curvature and let $u\in C^{2}(M)$ be a nonnegative finite energy solution of
\begin{equation}\label{eq-cp}
-\Delta u = u^{\frac{n+2}{n-2}}\quad\text{in }M.
\end{equation}
Then either $u\equiv 0$ on $M$ or $(M^n,g)$ is isometric to $\RR^n$ with the Euclidean metric and
$$
u(x)=\left(\frac{1}{a + b|x-x_0|^2}\right)^{\frac{n-2}{2}}
$$
for some $a,b>0$ with $1=n(n-2)ab$ and $x_0\in\RR^n$.
\end{theorem}

An immediate consequence of this result is the following

\begin{corollary}\label{c-sobmin} A complete, noncompact, nonflat, $n$-dimensional, $n\geq 3$, Riemannian manifold with nonnegative Ricci curvature does not admit any Sobolev minimizer of $S_g(M)$.
\end{corollary}

In particular, if the manifold is also Ricci flat it does not admit any Yamabe minimizer, i.e. a smooth function attaining the Yamabe constant
$$
Y(M,[g])=\inf_{0\not\equiv u\in C^\infty_0(M)}\frac{\int_M |\nabla u|^2\,dV_g+\tfrac{n-2}{4(n-1)}\int_M R_g\,u^2\,dV_g}{\left(\int_M u^{\frac{2n}{n-2}}\,dV_g\right)^{\frac{n-2}{n}}}.
$$
An alternative proof of Corollary \ref{c-sobmin} can be recovered using a recent result by Brendle \cite{bre}. In fact, if $S_g(M)=0$, then clearly it cannot be attained by any function in $D^{1,2}(M)$. If $S_g(M)>0$, then $(M^n,g)$ supports the Sobolev inequality and $S_g(M)$ is the best constant. Hence, by a result of Carron \cite{car}, $(M^n,g)$ has maximal volume growth, i.e. there exists $C>0$ such that
$$
\operatorname{Vol}_g B_\rho(x_0)\geq C \rho^n
$$
for every $x_0\in M$, $\rho>0$ and then one concludes using \cite{bre}.

\medskip

A result similar to Corollary \ref{c-sobmin} in the setting of Cartan-Hadamard manifolds (simply connected manifolds with nonpositive sectional curvature) has been recently obtained assuming the validity of an optimal isoperimetric inequality on $M$ (see \cite{mursoa}). We also point out a recent result obtained by Ciraolo, Figalli and Roncoroni \cite{cirfigron} concerning the classification of finite energy solutions to the critical (anisotropic) $p$-Laplace equation on convex cones of $\RR^n$, obtained via integral estimates and sharp a priori bounds.

\medskip

In the second theorem we consider the case $n\geq 3$ and solutions which may not have finite energy, but satisfy a suitable condition at infinity.

\begin{theorem}\label{t-cri2k1} Let $(M^n,g)$, $n\geq 3$, be a complete noncompact Riemannian manifold with nonnegative Ricci curvature and let $u\in C^{2}(M)$ be a nonnegative solution of
$$
-\Delta u = u^{\frac{n+2}{n-2}}\quad\text{in }M.
$$
If $n\geq4$ also assume that, outside a compact set of $M$,
$$u(x)\leq Cr(x)^{\alpha}\quad\text{for some }\alpha<-\tfrac{(n-2)(n-6)}{2(n-4)}.$$
Then either $u\equiv 0$ on $M$ or $(M^n,g)$ is isometric to $\RR^n$ with the Euclidean metric and
$$
u(x)=\left(\frac{1}{a + b|x-x_0|^2}\right)^{\frac{n-2}{2}}
$$
for some $a,b>0$ with $1=n(n-2)ab$ and $x_0\in\RR^n$.
\end{theorem}
We explicitly note that no assumption on the behavior of $u$ at $\infty$ is needed in Theorem \ref{t-cri2k1} if $n=3$. In particular on $\mathbb{R}^3$ we recover the full result by Caffarelli, Gidas and Spruck \cite{cafgidspr}. Moreover, we have the following

\begin{corollary} A complete, noncompact, nonflat, three-dimensional Riemannian manifold with nonnegative Ricci curvature does not admit any nonnegative, nontrivial solution of the critical equation
$$
-\Delta u = u^{5}.
$$
\end{corollary}
In case $n=4$, Theorem \ref{t-cri2k1} only needs that $u$ is bounded above by $r(x)$ to any power, i.e. $u$ has at most algebraic growth (with respect to the distance function), as $r(x)$ tends to $\infty$.
We also note that $\alpha>-\frac{n-2}{2}$ for every $n\geq4$, thus improving the classical results in $\mathbb{R}^n$ by Gidas-Ni-Nirenberg\cite{gidninir} and Obata \cite{oba} in any dimension $n\geq3$, where the authors assume that the solution decays as $|x|^{-(n-2)}$ at infinity.

\medskip

In our third theorem, we deal with the case $n=2$. We have the following result

\begin{theorem}\label{t-cri3k1} Let $(M^2,g)$ be a complete noncompact Riemannian surface with nonnegative scalar curvature. Let $u\in C^{2}(M)$ be a solution of
$$\begin{cases}-\Delta u= e^u \\ \int_M  e^u <+\infty \end{cases} $$
Assume that, outside a compact set of $M$,
$$
u(x)\geq -4\log r(x)-2\gamma\log\log r(x),\quad \gamma\in[0,1).
$$
Then $(M^2,g)$ is isometric to $\RR^2$ with the Euclidean metric and
$$
u(x)=\log\frac{1}{(a+b|x-x_0|^2)^2}
$$
for some $a,b>0$ with $1=8ab$ and $x_0\in\RR^2$.
\end{theorem}
In contrast to the Euclidean case, on a general Riemannian surface with nonnegative curvature, to the best of our knowledge, there is no result concerning the behavior of a solution $u$ at infinity. For this reason we have to assume an a priori lower bound. A better lower bound implying our condition was proved on $\RR^2$ by Chen and Li \cite{cheli2}, relying on \cite{bremer} and the explicit expression of the Green's function. In particular Theorem \ref{t-cri3k1} generalizes the result obtained in $\RR^2$ by Chen and Li \cite{cheli}.

\medskip

Our previous results are particular instances of more general theorems where we can allow for the presence of a nonnegative potential function $K$. The following theorem contains Theorem \ref{t-subk1} as a particular case.
\begin{theorem}\label{t-sub} Let $(M^n,g)$, $n\geq 2$, be a  complete noncompact Riemannian manifold with nonnegative Ricci curvature and let $u\in C^{2}(M)$ be a nonnegative solution of
$$
-\Delta u =K u^p\qquad\text{in }M
$$
with $1<p<p_c$, where
$$
p_c:=\begin{cases}+\infty &\text{if }n=2\\\frac{n+2}{n-2} &\text{if }n\geq 3\end{cases}
$$
and $0\not\equiv K\in C^2(M)$, $K\geq 0$. Moreover, if $n\geq 3$, we also assume $\Delta K\geq0$ on $M$, and if $n\geq 4$, we also assume $K(x)\geq \frac{C}{r(x)^\sigma}$ outside a compact set of $M$ for some $C>0$, $\sigma<\frac{2}{n-3}$. Then $$u\equiv 0\quad\text{on }M.$$
\end{theorem}

Theorem \ref{t-sub} improves the results by Gidas and Spruck \cite[Theorems 4.1, 6.1]{gidspr}, removing an assumption on $K$ and including the case $n=2$. Moreover, as it is clear from the proof, when $n\geq 4$ the lower bound on $K$ can be relaxed to
$$
K(x)\geq \frac{C}{r(x)^\sigma}\quad\text{with }\sigma<\sigma^*(n,p)
$$
for some explicit exponent $\sigma^*(n,p)\geq \frac{2}{n-3}$, see \eqref{eq-sigma}.

\medskip

The next two theorems extend Theorem \ref{t-crik1} and Theorem \ref{t-cri2k1}, respectively. In the first one we consider finite energy solutions, i.e. such that
$$
u\in D^{1,2}_K(M):=\left\{u: \,K\,u^{\frac{2n}{n-2}}\in L^1(M),\, |\nabla u|\in L^2(M)\right\},
$$
while in the second we deal with solutions with prescribed behavior at infinity.

\begin{theorem}\label{t-cri} Let $(M^n,g)$, $n\geq 3$, be a  complete noncompact Riemannian manifold with nonnegative Ricci curvature and let $u\in C^{2}(M)$ be a nonnegative finite energy solution of
$$
-\Delta u =K u^{\frac{n+2}{n-2}}\qquad\text{in }M
$$
with $0\not\equiv K\in C^2(M)$, $K\geq 0$ and $\Delta K\geq0$. Assume that, outside a compact set of $M$,
$$
(i)\,K(x)\leq C(1+r(x)^2),\qquad\text{or}\qquad (ii)\,|\nabla K(x)|\leq \frac{C}{r(x)}K(x).
$$
Then either $u\equiv 0$ on $M$ or $(M^n,g)$ is isometric to $\RR^n$ with the Euclidean metric and
$$
u(x)=\left(\frac{1}{a + b|x-x_0|^2}\right)^{\frac{n-2}{2}},\quad K\equiv n(n-2)ab
$$
for some $a,b>0$ and $x_0\in\RR^n$.
\end{theorem}
We note that, under suitable conditions on the potential $K$, assuming $K u^{\frac{2n}{n-2}}\in L^1(M)$ is sufficient  to deduce $u\in D^{1,2}_K(M)$, i.e. $u$ has finite energy. This is the case, in particular, if $K\equiv 1$. See Lemma \ref{lem3}.

\medskip

Similarly to Corollary \ref{c-sobmin}, we see that a complete, noncompact, nonflat, $n$-dimensional, $n\geq 3$, Riemannian manifold with nonnegative Ricci curvature does not admit any minimizer in $D^{1,2}_K(M)$ of the corresponding  weighted Sobolev quotient, with weight $K$ satisfying the assumptions of Theorem \ref{t-cri}.

\begin{theorem}\label{t-cri2} Let $(M^n,g)$, $n\geq 3$, be a  complete noncompact Riemannian manifold with nonnegative Ricci curvature and let $u\in C^{2}(M)$ be a nonnegative solution of
$$
-\Delta u =K u^{\frac{n+2}{n-2}}\qquad\text{in }M
$$
with $0\not\equiv K\in C^2(M)$, $K\geq 0$ and $\Delta K\geq0$. Assume that
$$|\nabla K(x)|\leq \frac{C}{r(x)}K(x)$$
outside a compact set of $M$. If $n\geq4$ also assume that
$$u\leq Cr^{\alpha}\quad\text{and}\quad K(x)\geq\frac{C}{r(x)^\sigma}$$
outside a compact set of $M$, for some
$$\alpha<\max\left\{-\tfrac{(n-2)(n-4)}{2(n-3)},-\tfrac{(n-2)^2\sigma+2(n-2)(n-6)}{4(n-4)}\right\}.$$
Then either $u\equiv 0$ on $M$ or $(M^n,g)$ is isometric to $\RR^n$ with the Euclidean metric and
$$
u(x)=\left(\frac{1}{a + b|x-x_0|^2}\right)^{\frac{n-2}{2}},\quad K\equiv n(n-2)ab
$$
for some $a,b>0$ and $x_0\in\RR^n$.
\end{theorem}

From this theorem, we have the following

\begin{corollary} A complete, noncompact, nonflat, three-dimensional Riemannian manifold with nonnegative Ricci curvature does not admit any nonnegative, nontrivial solution of the critical equation
$$
-\Delta u = Ku^{5},
$$
with $0\not\equiv K\in C^2(M)$, $K\geq 0$, $\Delta K\geq0$ and $$|\nabla K(x)|\leq \frac{C}{r(x)}K(x)$$
outside a compact set of $M$.
\end{corollary}

\medskip

In the proofs of Theorem \ref{t-sub} and \ref{t-cri2} we use Bishop-Gromov volume estimate which ensures that the volume of geodesic balls of radius $\rho$ grows at most as $C \rho^n$ as $\rho$ tends to infinity. As it is clear from the proofs, a slower rate of growth allows for weaker assumption on $u$ and/or $K$. We leave the details to the interested reader.

\medskip

Our last theorem deals with the case $n=2$ and generalizes Theorem \ref{t-cri3k1}.
\begin{theorem}\label{t-cri3} Let $(M^2,g)$ be a  complete noncompact Riemannian surface with nonnegative scalar curvature. Let $u\in C^{2}(M)$ be a solution of
$$\begin{cases}-\Delta u=K e^u \\ \int_M K e^u <+\infty \end{cases} $$
with $0\not\equiv K\in C^2(M)$, $K\geq 0$ and $\Delta K\geq0$. Assume that, outside a compact set of $M$,
$$
u(x)\geq -4\log r(x)-2\gamma\log\log r(x),\quad \gamma\in[0,1)
$$
and, for some $C>0$,
$$
(i)\,K(x)\leq C(1+r(x)^2)\,\,\text{and }u(x)\leq C \log r(x),\qquad\text{or}\qquad (ii)\,|\nabla K(x)|\leq \frac{C}{r(x)}K(x).
$$
Then $(M^2,g)$ is isometric to $\RR^2$ with the Euclidean metric and
$$
u(x)=\log\frac{1}{(a+b|x-x_0|^2)^2},\quad K\equiv 8ab
$$
for some $a,b>0$ and $x_0\in\RR^2$.
\end{theorem}
It is interesting to observe that Chen and Li in \cite{cheli2} exhibited the explicit radial solutions
$$
u_\alpha(r)=(2+\alpha)\log\frac{4}{4+r^2}
$$
in $\RR^2$ for
$$
K(r)=(2+\alpha)\left(\frac{4+r^2}{4}\right)^\alpha,
$$
for every $\alpha\in\RR$. Note that in these examples $\Delta K\geq 0$ in $\RR^2$ if and only if $\alpha\geq 0$, while our lower bound on $u_\alpha$ is satisfied if and only if $\alpha\leq 0$. In the same paper, the authors provide some sufficient conditions on $K$ ensuring the validity of upper and lower bounds for solutions on $\RR^2$.
\medskip

Finally, to the best of our knowledge, Theorems \ref{t-cri}, \ref{t-cri2} and \ref{t-cri3} are new even in the Euclidean setting. In the more general Riemannian setting, we point out some existence results of variational solutions for the Yamabe equation under conditions on the Yamabe constant and the Yamabe constant at infinity, see \cite{kim, wei}. We explicitly note that the Yamabe equation reduces to \eqref{eq-cp} when $n\geq 3$ and the manifold is scalar flat. Of course, the conditions in the cited references cannot hold on manifolds with nonnegative Ricci curvature.

\medskip

The paper is organized as follows: in Section \ref{s-1} we collect all the technical lemmas that we will need in the proofs of our main theorems; in Section \ref{s-sub} we prove Theorem \ref{t-sub} concerning nonexistence of nontrivial solutions in the subcritical case; in Sections \ref{s-cri} and \ref{s-cri2} we prove Theorems \ref{t-cri} and \ref{t-cri2} which deal with the critical case, when $n\geq 3$; finally in Section \ref{s-cri3} we prove Theorem \ref{t-cri3} on the critical equation on surfaces.

\

\section{Preliminary lemmas}\label{s-1}

\subsection{Part I} We collect here all technical lemmas that we will need in the study of equation \eqref{eq-eq}. We start showing the following key technical identity. We note that in the Euclidean case Bidaut-V\'eron and Raoux in \cite[Lemma~3.1]{bidrao} showed an estimate without including two extra terms that are crucial in proving our result in the critical case.

\begin{lemma}\label{l-BV}
Let $(M^n,g)$, $n\geq 2$, be a Riemannian manifold. For any positive function $w\in C^{2}(M)$, any nonnegative $\eta \in C^{2}_{c}(M)$ and any real numbers $d,m\in\RR$ such that $d\neq m+2$, the following identity holds:
\begin{align*}
&\frac{2(n-m)d-(n-1)(m^2+d^2)}{4n}\int_{M}\eta w^{m-2}|\nabla w|^{4}-\frac{n-1}{n}\int_{M}\eta w^m(\Delta w)^{2}\\
&\quad-\frac{2(n-1)m+(n+2)d}{2n}\int_{M}\eta w^{m-1}|\nabla w|^{2}\Delta w \\
&\quad+\frac{4}{(m+2-d)^2}\int_M \eta w^d \left| \mathring{\nabla}^2 w^{\frac{m+2-d}{2}}\right|^2+\frac{4}{(m+2-d)^2}\int_M \eta w^d \operatorname{Ric}\left(\nabla w^{\frac{m+2-d}{2}}, \nabla w^{\frac{m+2-d}{2}}\right)\\
&= \frac{m+d}{2}\int_{M}w^{m-1}|\nabla w|^{2}\langle\nabla w, \nabla \eta\rangle + \int_{M}w^m\Delta w \langle\nabla w, \nabla \eta\rangle+\frac12 \int_{M}w^m|\nabla w|^{2}\Delta\eta.
\end{align*}
\end{lemma}
\begin{proof} Let $f:=w^t$, for some $t\neq 0$. Then
$$
\nabla f = t w^{t-1}\nabla w,\qquad \Delta f=t w^{t-1}\Delta w + t(t-1)w^{t-2}|\nabla w|^2.
$$
Now we use the Bochner formula
\begin{align*}
\frac12 \Delta|\nabla f|^2&=|\nabla^2 f|^2+ \operatorname{Ric}(\nabla f,\nabla f)+\langle \nabla f,  \nabla \Delta f\rangle \\
&= \left|\mathring{\nabla}^2f\right|^2+ \frac{1}{n}(\Delta f)^2+ \operatorname{Ric}(\nabla f,\nabla f) +\langle \nabla f, \nabla\Delta f\rangle\\
&= \left|\mathring{\nabla}^2 w^t\right|^2+ \operatorname{Ric}(\nabla w^t,\nabla w^t)+\frac{t^2}{n}w^{2t-2}(\Delta w)^2 + \frac{t^2(t-1)^2}{n}w^{2t-4}|\nabla w|^4 \\
&\quad+ \frac{2t^2(t-1)}{n}w^{2t-3}|\nabla w|^2 \Delta w +t^2(t-1)w^{2t-3}|\nabla w|^2\Delta w\\
&\quad+ t^2w^{2t-2}\langle\nabla w,\nabla\Delta w\rangle+t^2(t-1)w^{2t-3}\langle\nabla w,\nabla |\nabla w|^2\rangle \\
&\quad+ t^2(t-1)(t-2)w^{2t-4}|\nabla w|^4.
\end{align*}
Let $d\in\RR$, multiplying the above equation by $t^{-2}\eta w^d$ and integrating over $M$, we obtain the following
\begin{align*}
&\left(\frac{(t-1)^2}{n}+(t-1)(t-2)\right)\int_M \eta w^{d+2t-4}|\nabla w|^4 +\frac{(t-1)(n+2)}{n}\int_M\eta w^{d+2t-3}|\nabla w|^2\Delta w \\
&\quad+\int_M\eta w^{d+2t-2}\langle\nabla w,\nabla\Delta w\rangle+\frac{1}{n}\int_M \eta w^{d+2t-2}(\Delta w)^2\\
&\quad+(t-1)\int_M \eta w^{d+2t-3}\langle\nabla w,\nabla|\nabla w|^2\rangle-\frac12\int_M \eta w^d \Delta\left(w^{2t-2}|\nabla w|^2\right)\\
&\quad+t^{-2}\int_M \eta w^d \left| \mathring{\nabla}^2 w^t\right|^2+t^{-2}\int_M \eta w^d \operatorname{Ric}(\nabla w^t, \nabla w^t) =0.
\end{align*}
Integrating by parts, we obtain
\begin{align*}
\int_M\eta w^{d+2t-2}\langle\nabla w,\nabla\Delta w\rangle &= -\int_M w^{d+2t-2}\Delta w \langle\nabla w,\nabla \eta\rangle \\
&\quad-(d+2t-2)\int_M \eta w^{d+2t-3}|\nabla w|^2\Delta w - \int_M \eta w^{d+2t-2}(\Delta w)^2,
\end{align*}
\begin{align*}
\int_M \eta w^{d+2t-3}\langle\nabla w,\nabla|\nabla w|^2\rangle &= -\int_M w^{d+2t-3}|\nabla w|^2 \langle\nabla w,\nabla \eta\rangle \\
&\quad-(d+2t-3)\int_M \eta w^{d+2t-4}|\nabla w|^4 - \int_M \eta w^{d+2t-3}|\nabla w|^2\Delta w,
\end{align*}
\begin{align*}
\int_M \eta w^d \Delta\left(w^{2t-2}|\nabla w|^2\right) &= \int_M w^{d+2t-2}|\nabla w|^2 \Delta \eta +2d\int_M w^{d+2t-3}|\nabla w|^2\langle\nabla w,\nabla\eta\rangle\\
&\quad+d(d-1)\int_M \eta w^{d+2t-4}|\nabla w|^4 +d \int_M \eta w^{d+2t-3}|\nabla w|^2\Delta w.
\end{align*}
Substituting in the above identity, rearranging terms and setting $m:=d+2t-2$ we conclude.
\end{proof}

In the next lemma we apply the previous identity to positive solutions of equation \eqref{eq-eq}.

\begin{lemma}\label{l-key}
Let $(M^n,g)$, $n\geq 2$, be a  Riemannian manifold. For any positive solution $u\in C^{2}(M)$ of $$-\Delta u=Ku^p,$$ with $p\in\RR$, $K\in C^2(M)$, any nonnegative $\psi \in C^{2}_{c}(M)$ and any real numbers $q\geq 2 $, $d,m\in\RR$ such that $d\neq m+2$, $m+p+1\neq0$ the following identity holds
\begin{align*}
&\alpha\int_{M}\psi^q u^{m-2}|\nabla u|^{4} + \beta \int_M \psi^qK u^{m+p-1}|\nabla u|^2+\frac{n-1}{n(m+p+1)}\int_M\psi^q\Delta K\, u^{m+p+1}\\
&\quad\quad+\frac{4}{(m+2-d)^2}\int_M \psi^q u^d \left| \mathring{\nabla}^2 u^{\frac{m+2-d}{2}}\right|^2+\frac{4}{(m+2-d)^2}\int_M \psi^q u^d \operatorname{Ric}\left(\nabla u^{\frac{m+2-d}{2}}, \nabla u^{\frac{m+2-d}{2}}\right)\\
&= \frac{n-2}{n}\int_M  Ku^{m+p}\langle \nabla u, \nabla\psi^q\rangle+\frac{m+d}{2}\int_{M}u^{m-1}|\nabla u|^{2}\langle\nabla u, \nabla \psi^q\rangle\\
&\quad\quad+\frac12 \int_{M}u^m|\nabla u|^{2}\Delta\psi^q+\frac{n-1}{n(m+p+1)}\int_M\Delta\psi^q K u^{m+p+1},
\end{align*}
where
$$
\alpha=\frac{2(n-m)d-(n-1)(m^2+d^2)}{4n}\qquad\text{and}\qquad \beta=\frac{(n+2)d-2(n-1)p}{2n}.
$$
\end{lemma}
\begin{proof}
Applying Lemma~\ref{l-BV} to $w=u$ and $\eta=\psi^q$, for any reals $d,m$ with $d\neq m+2$, we get
\begin{align*}
&\frac{2(n-m)d-(n-1)(m^2+d^2)}{4n}\int_{M}\psi^q u^{m-2}|\nabla u|^{4}-\frac{n-1}{n}\int_{M}\psi^q u^m(\Delta u)^{2}\\
&\quad\quad-\frac{2(n-1)m+(n+2)d}{2n}\int_{M}\psi^q u^{m-1}|\nabla u|^{2}\Delta u \\
&\quad\quad+\frac{4}{(m+2-d)^2}\int_M \psi^q u^d \left| \mathring{\nabla}^2 u^{\frac{m+2-d}{2}}\right|^2+\frac{4}{(m+2-d)^2}\int_M \psi^q u^d \operatorname{Ric}\left(\nabla u^{\frac{m+2-d}{2}}, \nabla u^{\frac{m+2-d}{2}}\right)\\
&= \frac{m+d}{2}\int_{M}u^{m-1}|\nabla u|^{2}\langle\nabla u, \nabla \psi^q\rangle + \int_{M}u^m\Delta u \langle\nabla u, \nabla \psi^q\rangle+\frac12 \int_{M}u^m|\nabla u|^{2}\Delta\psi^q.
\end{align*}
Hence, using the equation $\Delta u=-Ku^p$, we obtain
\begin{align}
&\frac{2(n-m)d-(n-1)(m^2+d^2)}{4n}\int_{M}\psi^q u^{m-2}|\nabla u|^{4}\nonumber\\
&\quad\quad+\frac{n-1}{n}\int_{M}\psi^q Ku^{m+p}\Delta u+\frac{2(n-1)m+(n+2)d}{2n}\int_{M}\psi^q Ku^{m+p-1}|\nabla u|^{2}\nonumber\\
&\quad\quad+\frac{4}{(m+2-d)^2}\int_M \psi^q u^d \left| \mathring{\nabla}^2 u^{\frac{m+2-d}{2}}\right|^2+\frac{4}{(m+2-d)^2}\int_M \psi^q u^d \operatorname{Ric}\left(\nabla u^{\frac{m+2-d}{2}}, \nabla u^{\frac{m+2-d}{2}}\right)\\
&= \frac{m+d}{2}\int_{M}u^{m-1}|\nabla u|^{2}\langle\nabla u, \nabla \psi^q\rangle-\int_{M}Ku^{m+p}\langle\nabla u, \nabla \psi^q\rangle+\frac12 \int_{M}u^m|\nabla u|^{2}\Delta\psi^q.\label{ineq333}
\end{align}
Integrating by parts in the first integral in the second line above, we get
\begin{align*}
&\frac{n-1}{n}\int_{M}\psi^q Ku^{m+p}\Delta u\\
&=-\frac{n-1}{n}\int_M  Ku^{m+p}\langle \nabla u, \nabla\psi^q\rangle-\frac{(m+p)(n-1)}{n}\int_M K\psi^q u^{m+p-1} |\nabla u|^2\\
&\,\,\,\,\,\,\,-\frac{n-1}{n}\int_M\psi^qu^{m+p}\langle\nabla K,\nabla u\rangle,
\end{align*}
and
\begin{align*}
&-\frac{n-1}{n}\int_M\psi^q u^{m+p}\langle\nabla K,\nabla u\rangle=-\frac{n-1}{n(m+p+1)}\int_M\psi^q\langle\nabla K,\nabla u^{m+p+1}\rangle\\
&\,\,\,\,\,\,\,=\frac{n-1}{n(m+p+1)}\int_M\langle\nabla\psi^q,\nabla K\rangle u^{m+p+1}+\frac{n-1}{n(m+p+1)}\int_M \psi^q\Delta K\,u^{m+p+1}\\
&\,\,\,\,\,\,\,=-\frac{n-1}{n(m+p+1)}\int_M\Delta\psi^q\,Ku^{m+p+1}-\frac{n-1}{n}\int_M Ku^{m+p}\langle\nabla\psi^q,\nabla u\rangle\\
&\,\,\,\,\,\,\,\,\,\,\,\,\,\,+\frac{n-1}{n(m+p+1)}\int_M \psi^q\Delta K\,u^{m+p+1}.
\end{align*}
Thus, substituting and setting
$$
I_1=\int_M  Ku^{m+p}\langle \nabla u, \nabla\psi^q\rangle,\quad I_2=\int_{M}u^{m-1}|\nabla u|^{2}\langle\nabla u, \nabla \psi^q\rangle,\quad I_3=\int_{M}u^m|\nabla u|^{2}\Delta\psi^q,
$$
equality~\eqref{ineq333} becomes
\begin{align*}
&\frac{2(n-m)d-(n-1)(m^2+d^2)}{4n}\int_{M}\psi^q u^{m-2}|\nabla u|^{4}+\frac{n-1}{n(m+p+1)}\int_M \psi^q\Delta K\,u^{m+p+1}\\
&\quad\quad-2\frac{n-1}{n}I_1-\frac{(m+p)(n-1)}{n}\int_M \psi^q Ku^{m+p-1} |\nabla u|^2\\
&\quad\quad+\frac{2(n-1)m+(n+2)d}{2n}\int_{M}\psi^q Ku^{m+p-1}|\nabla u|^{2}-\frac{n-1}{n(m+p+1)}\int_M\Delta\psi^q\,Ku^{m+p+1}\\
&\quad\quad+\frac{4}{(m+2-d)^2}\int_M \psi^q u^d \left| \mathring{\nabla}^2 u^{\frac{m+2-d}{2}}\right|^2+\frac{4}{(m+2-d)^2}\int_M \psi^q u^d \operatorname{Ric}\left(\nabla u^{\frac{m+2-d}{2}}, \nabla u^{\frac{m+2-d}{2}}\right)\\
&= \frac{m+d}{2}I_2-I_1+\frac{1}{2}I_3,
\end{align*}
Hence, rearranging and simplifying, we conclude
\begin{align}
&\alpha\int_{M}\psi^q u^{m-2}|\nabla u|^{4} + \beta \int_M \psi^qK u^{m+p-1}|\nabla u|^2+\frac{n-1}{n(m+p+1)}\int_M \psi^q\Delta K\,u^{m+p+1}\\
&\quad\quad+\frac{4}{(m+2-d)^2}\int_M \psi^q u^d \left| \mathring{\nabla}^2 u^{\frac{m+2-d}{2}}\right|^2+\frac{4}{(m+2-d)^2}\int_M \psi^q u^d \operatorname{Ric}\left(\nabla u^{\frac{m+2-d}{2}}, \nabla u^{\frac{m+2-d}{2}}\right)\\
&= \frac{n-2}{n}I_1+\frac{m+d}{2}I_2+\frac12 I_3+\frac{n-1}{n(m+p+1)}\int_M\Delta\psi^q\,Ku^{m+p+1},\label{eq-ab}
\end{align}
where
$$
\alpha=\frac{2(n-m)d-(n-1)(m^2+d^2)}{4n}\qquad\text{and}\qquad \beta=\frac{(n+2)d-2(n-1)p}{2n}.
$$
\end{proof}

In the subcritical case $1<p<p_c=\frac{n+2}{n-2}$ we neglect some nonnegative terms in Lemma \ref{l-key}, thus obtaining the following integral gradient estimate.

\begin{corollary}\label{cor3}
Let $(M^n,g)$, $n\geq 2$, be a  Riemannian manifold with $\operatorname{Ric}\geq0$. For any positive solution $u\in C^{2}(M)$ of $$-\Delta u=Ku^p,$$ with $p\in\RR$ , $K\in C^2(M)$, $\Delta K\geq0$, any nonnegative $\psi \in C^{2}_{c}(M)$ and any real numbers $q\geq 2 $, $d,m\in\RR$ such that $m+p+1>0$ the following estimate holds
\begin{align*}
&\alpha\int_{M}\psi^q u^{m-2}|\nabla u|^{4} + \beta \int_M \psi^qK u^{m+p-1}|\nabla u|^2+\\
&\leq \frac{n-2}{n}\int_M  Ku^{m+p}\langle \nabla u, \nabla\psi^q\rangle+\frac{m+d}{2}\int_{M}u^{m-1}|\nabla u|^{2}\langle\nabla u, \nabla \psi^q\rangle\\
&\quad\quad+\frac12 \int_{M}u^m|\nabla u|^{2}\Delta\psi^q+\frac{n-1}{n(m+p+1)}\int_M\Delta\psi^q K u^{m+p+1},
\end{align*}
where
$$
\alpha=\frac{2(n-m)d-(n-1)(m^2+d^2)}{4n}\qquad\text{and}\qquad \beta=\frac{(n+2)d-2(n-1)p}{2n}.
$$
\end{corollary}

For $n\geq 3$, in the critical case $p=p_c=\frac{n+2}{n-2}$, by choosing
$$
m=-\frac{2}{n-2},\quad d=\frac{2(n-1)}{n-2}
$$
in Lemma \ref{l-key}, we get $\alpha=\beta=0$, thus obtaining the following identity.
\begin{corollary}\label{cor1}
Let $(M^n,g)$, $n\geq 3$, be a  Riemannian manifold. For any positive solution $u\in C^{2}(M)$ of $$-\Delta u=Ku^{\frac{n+2}{n-2}},$$ with $K\in C^2(M)$, any nonnegative $\psi \in C^{2}_{c}(M)$ and any real number $q\geq 2 $, the following identity holds
\begin{align*}
&\frac{(n-2)^2}{4}\int_M \psi^q u^\frac{2(n-1)}{n-2} \left| \mathring{\nabla}^2 u^{-\frac{2}{n-2}}\right|^2+\frac{(n-2)^2}{4}\int_M \psi^q u^\frac{2(n-1)}{n-2} \operatorname{Ric}\left(\nabla u^{-\frac{2}{n-2}}, \nabla u^{-\frac{2}{n-2}}\right)\\
&\quad\quad+\frac{n-2}{2n}\int_M \psi^q\Delta K\,u^{2\frac{n-1}{n-2}}\\
&= \frac{n-2}{n}\int_M  Ku^{\frac{n}{n-2}}\langle \nabla u, \nabla\psi^q\rangle+\int_{M}u^{-\frac{n}{n-2}}|\nabla u|^{2}\langle\nabla u, \nabla \psi^q\rangle+\frac12 \int_{M}u^{-\frac{2}{n-2}}|\nabla u|^{2}\Delta\psi^q\\
&\quad\quad+\frac{n-2}{2n}\int_M\Delta\psi^q\,Ku^{2\frac{n-1}{n-2}},
\end{align*}
\end{corollary}
In the next lemma we estimate some of the terms in the right hand side of the above identity in terms of controlled quantities. This inequality will be useful in the critical case.
\begin{lemma}\label{lem1}
 Let $(M^n,g)$, $n\geq 3$, be a  Riemannian manifold. For any positive solution $u\in C^{2}(M)$ of $$-\Delta u=Ku^{\frac{n+2}{n-2}},$$ with $K\in C^0(M)$, any nonnegative $\psi \in C^{2}_{c}(M)$ and any real numbers $q\geq 2$, $\eps>0$ the following estimate holds
\begin{align*}
&-\frac{1}{n}\int_M  Ku^{\frac{n}{n-2}}\langle \nabla u, \nabla\psi^q\rangle+\int_{M}u^{-\frac{n}{n-2}}|\nabla u|^{2}\langle\nabla u, \nabla \psi^q\rangle+\frac12 \int_{M}u^{-\frac{2}{n-2}}|\nabla u|^{2}\Delta\psi^q \\
&\leq \frac{q(n-2)^2\eps}{8}\int_M \psi^q u^\frac{2(n-1)}{n-2} \left| \mathring{\nabla}^2 u^{-\frac{2}{n-2}}\right|^2+\frac{q}{2\eps}\int_{M}\psi^{q-2}u^{-\frac{2}{n-2}}|\nabla u|^{2}|\nabla \psi|^2.
\end{align*}
\end{lemma}
\begin{proof} Let
$$
I_1= \int_M  Ku^{\frac{n}{n-2}}\langle \nabla u, \nabla\psi^q\rangle,\,\,\,\, I_2=\int_{M}u^{-\frac{n}{n-2}}|\nabla u|^{2}\langle\nabla u, \nabla \psi^q\rangle,\,\,\,\, I_3=\int_{M}u^{-\frac{2}{n-2}}|\nabla u|^{2}\Delta\psi^q.
$$
Integrating by parts, we obtain
\begin{align*}
I_2 &= \frac{(n-2)^3}{8(n-1)} \int_{M}|\nabla u^{-\frac{2}{n-2}}|^{2}\langle\nabla u^\frac{2(n-1)}{n-2}, \nabla \psi^q\rangle\\
&= -\frac{(n-2)^3}{8(n-1)} \int_{M} u^\frac{2(n-1)}{n-2} \langle \nabla |\nabla u^{-\frac{2}{n-2}}|^{2}, \nabla \psi^q\rangle-\frac{(n-2)^3}{8(n-1)} \int_{M} u^\frac{2(n-1)}{n-2}|\nabla u^{-\frac{2}{n-2}}|^{2} \Delta \psi^q\\
&=-\frac{(n-2)^3}{4(n-1)} \int_{M} u^\frac{2(n-1)}{n-2}  \nabla^2 u^{-\frac{2}{n-2}}\left(\nabla u^{-\frac{2}{n-2}}, \nabla\psi^q\right)-\frac{n-2}{2(n-1)} I_3 \\
&= -\frac{(n-2)^3}{4(n-1)} \int_{M} u^\frac{2(n-1)}{n-2}  \mathring{\nabla}^2 u^{-\frac{2}{n-2}}\left(\nabla u^{-\frac{2}{n-2}}, \nabla\psi^q\right)\\
&\quad-\frac{(n-2)^3}{4n(n-1)} \int_{M} u^\frac{2(n-1)}{n-2}  \Delta u^{-\frac{2}{n-2}}\langle \nabla u^{-\frac{2}{n-2}}, \nabla\psi^q\rangle-\frac{n-2}{2(n-1)} I_3.
\end{align*}
Since
$$
\Delta u^{-\frac{2}{n-2}}= \frac{2}{n-2}Ku^{\frac{2}{n-2}}+\frac{2n}{(n-2)^2}u^{-\frac{2(n-1)}{n-2}}|\nabla u|^2,
$$
using Young's inequality on the first integral in the last equality, we obtain
\begin{align*}
I_2 &\leq \frac{q(n-2)^3\eps}{8(n-1)} \int_M \psi^q u^\frac{2(n-1)}{n-2} \left| \mathring{\nabla}^2 u^{-\frac{2}{n-2}}\right|^2+\frac{q(n-2)^3}{8(n-1)\eps} \int_M u^{\frac{2(n-1)}{n-2}}\psi^{q-2}|\nabla u^{-\frac{2}{n-2}}|^2 |\nabla \psi|^2 \\
&\quad-\frac{(n-2)^3}{4n(n-1)} \int_{M} u^\frac{2(n-1)}{n-2}  \left(\frac{2}{n-2}Ku^{\frac{2}{n-2}}+\frac{2n}{(n-2)^2}u^{-\frac{2(n-1)}{n-2}}|\nabla u|^2\right)\langle \nabla u^{-\frac{2}{n-2}}, \nabla\psi^q\rangle\\
&\quad-\frac{n-2}{2(n-1)} I_3\\
&= \frac{q(n-2)^3\eps}{8(n-1)} \int_M \psi^q u^\frac{2(n-1)}{n-2} \left| \mathring{\nabla}^2 u^{-\frac{2}{n-2}}\right|^2+\frac{q(n-2)^3}{8(n-1)\eps} \int_M u^{\frac{2(n-1)}{n-2}}\psi^{q-2}|\nabla u^{-\frac{2}{n-2}}|^2 |\nabla \psi|^2 \\
&\quad+\frac{n-2}{n(n-1)}I_1+\frac{1}{n-1}I_2-\frac{n-2}{2(n-1)} I_3.
\end{align*}
Rearranging terms, we conclude.
\end{proof}

\subsection{Part II} We collect here all technical lemmas that we will need in the study of equation \eqref{eq-eq2}. We start showing the following key technical identity, which is the counterpart of Lemma \ref{l-BV}.

\begin{lemma}\label{lem4}
Let $(M^n,g)$, $n\geq 2$, be a Riemannian manifold. For any function $w\in C^{2}(M)$, any nonnegative $\eta \in C^{2}_{c}(M)$ and any real numbers $d,m\in\RR\setminus\{0\}$, the following identity holds:
\begin{align*}
&-\left(\frac{(n-1)m^2}{n}+md+\frac{d^2}{2}\right)\int_{M}\eta e^{(2m+d)w}|\nabla w|^{4}-\frac{n-1}{n}\int_{M}\eta e^{(2m+d)w}(\Delta w)^{2}\\
&\quad\quad-\left(\frac{2(n-1)m}{n}+\frac{3}{2}d\right)\int_{M}\eta e^{(2m+d)w}|\nabla w|^{2}\Delta w \\
&\quad\quad+\frac{1}{m^2}\int_M \eta e^{dw} \left| \mathring{\nabla}^2 e^{mw}\right|^2+\frac{1}{m^2}\int_M \eta e^{dw} \operatorname{Ric}\left(\nabla e^{mw}, \nabla e^{mw}\right)\\
&= (m+d)\int_{M}e^{(2m+d)w}|\nabla w|^{2}\langle\nabla w, \nabla \eta\rangle + \int_{M}e^{(2m+d)w}\Delta w \langle\nabla w, \nabla \eta\rangle+\frac12 \int_{M}e^{(2m+d)w}|\nabla w|^{2}\Delta\eta.
\end{align*}
\end{lemma}
\begin{proof}
  For $f=e^{mw}$, we have
  $$\nabla f=me^{mw}\nabla w\qquad\Delta f=me^{mw}\Delta w+m^2e^{mw}|\nabla w|^2.$$
  We use the Bochner formula
  \begin{align*}
  \frac{1}{2}\Delta|\nabla f|^2&=\left|\nabla^2f\right|^2+\operatorname{Ric}(\nabla f,\nabla f)+\langle \nabla f,\nabla\Delta f\rangle\\
  &=\left|\mathring{\nabla}^2f\right|^2+\frac{1}{n}(\Delta f)^2+\operatorname{Ric}(\nabla f,\nabla f)+\langle \nabla f,\nabla\Delta f\rangle\\
  &=\left|\mathring{\nabla}^2 e^{mw} \right|^2+\operatorname{Ric}(\nabla e^{mw},\nabla e^{mw})+\frac{m^2}{n}e^{2mw}(\Delta w)^2+\frac{m^4}{n}e^{2mw}|\nabla w|^4\\
  &\quad+\frac{2m^3}{n}e^{2mw}|\nabla w|^2\Delta w+m^3e^{2mw}\Delta w|\nabla w|^2 +m^2e^{2mw}\langle\nabla w,\nabla\Delta w\rangle\\
  &\quad+ m^4e^{2mw}|\nabla w|^4+m^3e^{2mw}\langle\nabla w,\nabla|\nabla w|^2\rangle
  \end{align*}
  Now we multiply by $m^{-2}\eta e^{dw}$ and integrate over $M$ to obtain
\begin{align}
  \nonumber&m^{-2}\int_M\eta e^{dw}\left|\mathring{\nabla}^2 e^{mw}\right|^2+m^{-2}\int_M\eta e^{dw}\operatorname{Ric}(\nabla e^{mw},\nabla e^{mw})+
  \frac{1}{n}\int_M\eta e^{(2m+d)w}(\Delta w)^2\\
  \label{eq-1}&\,\,\,\,\,+\frac{(n+1)m^2}{n}\int_M\eta e^{(2m+d)w}|\nabla w|^4+\frac{(2+n)m}{n}\int_M\eta e^{(2m+d)w}|\nabla w|^2\Delta w \\
  \nonumber&\,\,\,\,\,+\int_M\eta e^{(2m+d)w}\langle\nabla w,\nabla\Delta w\rangle+m\int_M\eta e^{(2m+d)w}\langle\nabla w,\nabla|\nabla w|^2\rangle\\
  \nonumber&\,\,\,\,\,-\frac{1}{2}\int_M\eta e^{dw}\Delta\left(e^{2mw}|\nabla w|^2\right)=0.
\end{align}
Integrating by parts
\begin{align*}
&\int_M\eta e^{(2m+d)w}\langle\nabla w,\nabla\Delta w\rangle\\
&\quad=-\int_Me^{(2m+d)w}\Delta w \langle\nabla\eta,\nabla w\rangle-(2m+d)\int_M\eta e^{(2m+d)w}\Delta w|\nabla w|^2-
\int_M\eta e^{(2m+d)w}(\Delta w)^2
\end{align*}
\begin{align*}
  & \int_M \eta e^{(2m+d)w}\langle\nabla w,\nabla|\nabla w|^2\rangle \\
  &\quad= -\int_Me^{(2m+d)w}|\nabla w|^2\langle\nabla\eta,\nabla w\rangle-\int_M\eta e^{(2m+d)w}|\nabla w|^2\Delta w-(2m+d)\int_M\eta e^{(2m+d)w}|\nabla w|^4
\end{align*}
and
\begin{align*}
&\int_M\eta e^{dw}\Delta\left(e^{2mw}|\nabla w|^2\right)\\
&\quad=\int_M e^{(2m+d)w}\Delta\eta|\nabla w|^2+2d\int_M e^{(2m+d)w}|\nabla w|^2\langle\nabla w,\nabla\eta\rangle\\
&\quad\,\,\,\,\, +d^2\int_M \eta e^{(2m+d)w}|\nabla w|^4+d\int_M\eta e^{(2m+d)w}|\nabla w|^2\Delta w
\end{align*}
Substituting these identities in \eqref{eq-1} and rearranging terms we obtain the result.
\end{proof}

We apply the previous identity to solutions of equation \eqref{eq-eq2}.

\begin{corollary}\label{cor2}
  Let $(M^2,g)$ be a Riemannian surface, with scalar curvature $R$. For any solution $w\in C^{2}(M)$ of $$-\Delta u=K e^u$$ with $K\in C^2(M)$, any nonnegative $\psi \in C^{2}_{c}(M)$ and any real number $q\geq 2 $ the following identity holds
\begin{align*}
8\int_M \psi^q e^{\frac{u}{2}} \left| \mathring{\nabla}^2 e^{-\frac{u}{2}}\right|^2&+4\int_M \psi^q e^{\frac{u}{2}} R \left|\nabla e^{-\frac{u}{2}}\right|^2+2\int_M \psi^q e^{\frac u2}\Delta K \\
&= 2\int_{M}K e^\frac{u}{2} \Delta \psi^q +\int_{M}e^{-\frac{u}{2}}|\nabla u|^{2}\Delta\psi^q.
\end{align*}
\end{corollary}
\begin{proof}
  We start from Lemma \ref{lem4} with $w=u$, $\eta=\psi^q$. Integrating by parts, using the equation on $u$, we obtain
  \begin{align*}
     \int_M \psi^q e^{(2m+d)u}(\Delta u)^2&=-\int_M \psi^q K e^{(2m+d+1)u}\Delta u\\
    &= \int_M K e^{(2m+d+1)u}\langle\nabla\psi^q,\nabla u\rangle+(2m+d+1)\int_M \psi^qK e^{(2m+d+1)u}|\nabla u|^2 \\
    &\quad + \int_M \psi^q e^{(2m+d+1)u} \langle \nabla K , \nabla u\rangle.
  \end{align*}
Assuming $2m+d+1\neq 0$, we have
\begin{align*}
\int_M &\psi^q e^{(2m+d+1)u} \langle \nabla K , \nabla u\rangle =\frac{1}{2m+d+1}\int_M \psi^q \langle\nabla K,\nabla e^{(2m+d+1)u}\rangle\\
&=-\frac{1}{2m+d+1}\int_M \psi^q e^{(2m+d+1)u} \Delta K-\frac{1}{2m+d+1}\int_M e^{(2m+d+1)u} \langle \nabla K,\nabla \psi^q\rangle\\
&=-\frac{1}{2m+d+1}\int_M \psi^q e^{(2m+d+1)u} \Delta K+\frac{1}{2m+d+1}\int_M K e^{(2m+d+1)u} \Delta \psi^q \\
&\quad+\int_M K e^{(2m+d+1)u}\langle \nabla u,\nabla \psi^q\rangle.
\end{align*}
Hence
\begin{align*}
&-\left(\frac{(n-1)m^2}{n}+md+\frac{d^2}{2}\right)\int_{M}\psi^q e^{(2m+d)u}|\nabla u|^{4}+\left(\frac{n+2}{2n}d-\frac{n-1}{n}\right)\int_{M}\psi^q K e^{(2m+d+1)u}|\nabla u|^{2} \\
&\quad\quad+\frac{1}{m^2}\int_M \psi^q e^{du} \left| \mathring{\nabla}^2 e^{mu}\right|^2+\frac{1}{m^2}\int_M \psi^q e^{du} \operatorname{Ric}\left(\nabla e^{mu}, \nabla e^{mu}\right)\\
&\quad\quad+\frac{n-1}{n(2m+d+1)}\int_M \psi^q e^{(2m+d+1)u} \Delta K\\
&= (m+d)\int_{M}e^{(2m+d)u}|\nabla u|^{2}\langle\nabla u, \nabla \psi^q\rangle +\frac{n-2}{n} \int_{M}K e^{(2m+d+1)u} \langle\nabla u, \nabla \psi^q\rangle\\
&\quad+\frac12 \int_{M}e^{(2m+d)u}|\nabla u|^{2}\Delta\psi^q+\frac{n-1}{n(2m+d+1)}\int_M K e^{(2m+d+1)u} \Delta \psi^q.
\end{align*}
  We have $n=2$, hence $\operatorname{Ric}=\frac{1}{2}Rg$, and we choose $m=-\frac{1}{2}$, $d=\frac{1}{2}$ so that
  \begin{align*}
  4\int_M \psi^q e^\frac{u}{2} \left| \mathring{\nabla}^2 e^{-\frac{u}{2}}\right|^2&+2\int_M \psi^q e^\frac{u}{2} R\left|\nabla e^{-\frac{u}{2}}\right|^2+\int_M \psi^q e^{\frac u2}\Delta K\\
   &= \int_{M}K e^{\frac{u}{2}} \Delta \psi^q+\frac12 \int_{M}e^{-\frac{u}{2}}|\nabla u|^{2}\Delta\psi^q.
  \end{align*}
\end{proof}

\subsection{Part III} We collect here some general lemmas that we will need in the proofs of our main theorems.
The first is a lower bound for positive superharmonic functions on Riemannian manifold with nonnegative Ricci curvature. We include a proof for completeness.
\begin{lemma}\label{lem2} Let $(M^n,g)$, $n\geq 2$, be a complete Riemannian manifold with nonnegative Ricci curvature and let $u\in C^2(M)$ be a positive superharmonic function outside a compact set of $M$, i.e. $\Delta u \leq 0$ on $M\setminus\Omega$ with $\Omega$ compact. Then there exist positive constants $\rho,A>0$ such that
$$
u(x) \geq \frac{A}{r(x)^{n-2}}\quad\text{on }B_\rho^c.
$$
\end{lemma}
\begin{proof} Let $\rho>0$ be such that $\Delta u\leq0$ on $B_\rho^c$. On $B_\rho^c$ let
$$
v(x):=u(x)- \frac{A}{r(x)^{\alpha}},
$$
with $\alpha=n-2$ if $n\geq3$, $\alpha>0$ if $n=2$ and where $A:=\rho^\alpha\min_{\partial B_\rho} u>0$. Then $v\geq 0$ on $\partial B_\rho$ and
$$
\Delta v\leq 0,
$$
since $\operatorname{Ric}\geq0$ implies $\Delta r \leq \frac{n-1}{r}$ weakly on $B_\rho^c$ by classical Laplacian comparison. Since $\liminf_{r(x)\to\infty} v (x) \geq 0$, if $\inf_{B_\rho^c} v<0$, then $v$ attains its negative absolute minimum at a point in $\overline{B_\rho}^c$. By the strong maximum principle, then $v$ must be constant and negative on its domain, a contradiction. Thus $v\geq 0$ in $B^c_\rho$. The proof is complete if $n\geq3$. If $n=2$, for every $x\in B_\rho^c$ we have $$u(x)\geq \frac{\rho^\alpha\min_{\partial B_\rho}u}{r(x)^\alpha}$$ for every $\alpha>0$, and we conclude passing to the limit as $\alpha$ tends to $0$.
\end{proof}

The next lemma shows that, under suitable assumptions on the potential $K$, subsolutions of the equation \eqref{eq-eq} with $u^{p+1}|K|\in L^1(M)$ automatically have finite energy.

\begin{lemma}\label{lem3} Let $(M^n,g)$, $n\geq 2$, be a  Riemannian manifold and let  $u\in C^{2}(M)$ be a nonnegative solution of $$-\Delta u\leq Ku^p,$$ with $1<p\leq p_c$ and $K\in C^0(M)$. If $K>0$ almost everywhere outside a compact set of $M$, $|K|u^{p+1}\in L^1(M)$ and $$\int_{B_{2\rho}\setminus B_\rho}K^{-\frac{2}{p-1}}\leq C\rho^{2\frac{p+1}{p-1}}\quad\text{ for every }\rho\text{ large enough,}$$ then $|\nabla u|\in L^2(M)$, i.e. $u\in D^{1,2}_{|K|}(M)$.
\end{lemma}
\begin{proof} For any $\rho>1$, let $\psi\in C^2_c(M)$ be such that $\psi\equiv 1$ in $B_\rho$, $\psi\equiv 0$ in $B_{2\rho}^c$, $0\leq \psi\leq 1$ on $M$ and $\psi$ satisfies
$$
|\nabla \psi|^{2} \leq C \rho^{-2}\quad\text{in }B_{2\rho}\setminus B_{\rho}.
$$
Then, for every $0<\eps<1$ we obtain
\begin{align*}
\int_M K\psi^2 u^{p+1} &\geq -\int_M \psi^2 u \Delta u  \\
&\geq \int_M \psi^2 |\nabla u|^2 -2 \int_M \psi u |\nabla u||\nabla \psi|\\
&\geq (1-\eps)\int_M \psi^2 |\nabla u|^2 - \frac{1}{\eps} \int_M u^2 |\nabla \psi|^2.
\end{align*}
For every large enough $\rho$ we have
\begin{align*}
\int_M u^2 |\nabla \psi|^2 &\leq \frac{C}{\rho^2}\left(\int_{B_{2\rho}} |K|u^{p+1} \right)^{\frac{2}{p+1}}\left(\int_{B_{2\rho}\setminus B_\rho}K^{-\frac{2}{p-1}}\right)^\frac{p-1}{p+1}\leq C  \left(\int_{B_{2\rho}} |K|u^{p+1} \right)^{\frac{2}{p+1}}.
\end{align*}
By the previous estimate, we obtain
$$
\int_{B_{\rho}}|\nabla u|^2 \leq C \left(\int_{B_{2\rho}} |K|u^{p+1} \right)^{\frac{2}{p+1}}+ C \int_{B_{2\rho}}  |K|u^{p+1}.
$$
Passing to the limit as $\rho\to\infty$, since $u^{p+1}|K|\in L^1(M)$, we obtain the result.
\end{proof}
Note that if $$K\geq\frac{C}{(1+r(x))^\lambda}\quad\text{ on }M$$
  for some $\lambda\leq\frac{(n+2)-p(n-2)}{2}$, then by Bishop-Gromov volume comparison theorem we have
  $$\int_{B_{2\rho}\setminus B_\rho}K^{-\frac{2}{p-1}}\leq C\rho^{n+\lambda\frac{2}{p-1}}\leq C\rho^{2\frac{p+1}{p-1}}$$ for every $\rho>0$. In particular for the critical equation with $p=p_c=\frac{n+2}{n-2}$, the condition on $K$ reads as
$$K\geq C>0\quad\text{on } M.$$
Note that these conditions are satisfied if $K$ is constant and positive on $M$.

\medskip

We finally recall the following characterization of conformal gradient vector fields which can be deduced from general results in \cite{tas} and \cite{catmanmaz} and which will be used when dealing with critical equations.
\begin{lemma}\label{lem5}
  Let $(M^n,g)$, $n\geq2$, be a complete noncompact Riemannian manifold with nonnegative Ricci curvature and assume there exists a nonconstant function $f$ on $M$ such that $$\mathring{\nabla}^2f\equiv0,\quad \operatorname{Ric}(\nabla f,\nabla f)\equiv0\qquad\text{ on }M.$$
  Then $(M^n,g)$ is isometric to either
   \begin{itemize}
     \item[i)] a direct product $(\mathbb{R}\times N^{n-1}, dr^2+g_N)$, where $(N^{n-1},g_N)$ is a $(n-1)$-dimensional complete Riemannian manifold with nonnegative Ricci curvature and $f=f(r)=ar+b$ for some $a,b\in\mathbb{R}$ with $a\neq0$;
     \item[ii)] $\mathbb{R}^n$ with the Euclidean metric and $f=f(x)=a|x-x_0|^2+b$ for some $x_0\in\mathbb{R}^n$ and $a,b\in\mathbb{R}$ with $a\neq0$.
   \end{itemize}
      \end{lemma}

\

\section{Proof of Theorem \ref{t-sub}}\label{s-sub}

\begin{proof} We start considering the case of a Riemannian surface $(M^2,g)$ with nonnegative scalar curvature $R$. As we already observed, by Bishop-Gromov volume estimate $(M^2, g)$ must be parabolic (see for instance \cite{gri}) and therefore every positive superharmonic function is constant. Thus any solution $u\geq 0$ of \eqref{eq-eq} is zero. We provide here a self-contained proof. Suppose by contradiction that $u$ is not identically $0$, then by the strong maximum principle $u>0$ on $M$. Let $\psi\in Lip_c(M)$ be a nonnegative cutoff function, $\delta\in(-p,-1)$ and $\varepsilon=|\delta|$, then we have
\begin{align}
  \int_MK\psi^2 u^{p+\delta}&=-\int_M\psi^2 u^{\delta}\Delta u=2\int_M\psi\langle\nabla\psi,\nabla u\rangle u^{\delta}+\delta\int_M\psi^2 u^{\delta-1}|\nabla u|^2\\
 \label{eq-4}&\leq\frac{1}{\varepsilon}\int_M|\nabla\psi|^2u^{\delta+1}+(\varepsilon+\delta)\int_M\psi^2 u^{\delta-1}|\nabla u|^2=\frac{1}{\varepsilon}\int_M|\nabla\psi|^2u^{\delta+1}
 \end{align}
Let $\Phi\in C^1(\mathbb{R})$ be a nonnegative function such that $0\leq\Phi\leq1$, $\Phi'\leq0$, $\Phi=1$ on $(-\infty,1]$, $\Phi=0$ on $[2,-\infty)$ and $|\Phi'|\leq C$ for some positive constant $C$. For any $\rho>1$ define $\psi=\Phi\left(\frac{2\log r}{\log \rho}\right)$, then we have $\psi\equiv1$ on $B_{\sqrt\rho}$, $\psi\equiv0$ on $B_{\rho}^c$ and $|\nabla\psi|\leq\frac{C}{\log\rho\, r}$ on $B_{\rho}\setminus B_{\sqrt\rho}$.
With this choice of $\psi$, \eqref{eq-4} and Lemma \ref{lem2} yield
$$
  \int_MK\psi u^{p+\delta}\leq \frac{C}{(\log\rho)^2}\int_{B_{\rho}\setminus B_{\sqrt\rho}}\frac{1}{r^2}.
$$
Let $V(r)=\operatorname{Vol}(B_r)$ and $S(r)=\operatorname{meas}(\partial B_r)$ be the area of $B_r$ and the length of $\partial B_r$ respectively. Then by the coarea formula $V'=S$, and using the coarea formula, an integration by parts and Bishop-Gromov volume comparison theorem we obtain
\begin{align*}
  \int_{B_{\rho}\setminus B_{\sqrt\rho}}\frac{1}{r^2} & =\int_{\sqrt\rho}^\rho \frac{1}{r^2} S(r)\,dr=\left[\frac{1}{r^2} V(r)\right]_{\sqrt\rho}^\rho+2\int_{\sqrt\rho}^\rho\frac{1}{r^3} V(r)\,dr\\
   &\leq C\left(1+\int_{\sqrt\rho}^\rho\frac{1}{r}\right)\leq C(1+\log \rho).
\end{align*}
Hence we have
$$\int_MK\psi u^{p+\delta}\leq \frac{C}{\log\rho}$$
and passing to the limit as $\rho$ tends to $\infty$ we deduce that $u=0$ on the set $\{x\in M\,|\,K(x)>0\}$. This is a contradiction, since we assumed $u>0$ on $M$.

Now we consider the case when $(M^n,g)$, $n\geq 3$, is a Riemannian manifold with nonnegative Ricci curvature. Let $u\in C^{2}(M)$ be a nonnegative solution of \eqref{eq-eq} with $1<p<p_c=\frac{n+2}{n-2}$. By the strong maximum principle, $u\equiv 0$ or $u>0$ on $M$. We assume by contradiction that $u>0$ on $M$ and we apply Corollary \ref{cor3} with $\psi\in C^2_c(M)$, $0\leq \psi \leq 1$ on $M$, $q>4$ and with $m,d$ satisfying
$$m=-2\quad\text{ and }d=2+\frac{2}{5}p\qquad\text{ if }n=3,$$ or
$$2\frac{n-1}{n+2}p<d<2\frac{n-1}{n-2},$$ and
$$\frac{-d-\sqrt{dn(2(n-1)-d(n-2))}}{n-1}<m<\frac{-d+\sqrt{dn(2(n-1)-d(n-2))}}{n-1}$$
when $n\geq4$. Note that the condition on $d$ is meaningful, since $p<p_c=\frac{n+2}{n-2}$, and that with these choices $m>-2$ when $n\geq4$.
Then we have
\begin{align}\label{eq-ciao}
&\alpha\int_{M}\psi^q u^{m-2}|\nabla u|^{4} + \beta \int_M \psi^qK u^{m+p-1}|\nabla u|^2\\
&\qquad\leq\frac{n-2}{n}I_1+\frac{m+d}{2}I_2+\frac12 I_3+\frac{(n-1)}{n(m+p+1)}\int_M\Delta\psi^q\,Ku^{m+p+1},
\end{align}
where
$$
I_1=\int_M  Ku^{m+p}\langle \nabla u, \nabla\psi^q\rangle,\quad I_2=\int_{M}u^{m-1}|\nabla u|^{2}\langle\nabla u, \nabla \psi^q\rangle,\quad I_3=\int_{M}u^m|\nabla u|^{2}\Delta\psi^q
$$
and with $\alpha,\beta>0$ for every $n\geq3$. We now bound the integrals $I_1, I_2$ and $I_3$. By using Young's inequality, for any $\eps>0$, there exists $C(\eps)>0$ such that
\begin{align}\label{eq-I1}
I_1 &\leq \eps \int_{M}\psi^{q}Ku^{m+p-1}|\nabla u|^{2} + C(\eps)\int_{M}\psi^{q-2}Ku^{m+p+1}|\nabla \psi|^{2},\\
I_2 &\leq \eps \int_{M} \psi^{q}u^{m-2}|\nabla u|^{4}+C(\eps)\int_{M}\psi^{q-4}u^{m+2}|\nabla \psi|^{4},\\
I_3 &= q(q-1)\int_{M}\psi^{q-2}u^{m}|\nabla u|^{2}|\nabla \psi|^{2}+q\int_{M}\psi^{q-1}u^{m}|\nabla u|^{2}\Delta \psi\nonumber \\
&\leq \eps \int_{M}\psi^{q}u^{m-2}|\nabla u|^{4}+C(\eps)\int_{U}\psi^{q-4}u^{m+2}|\nabla \psi|^{4}+C(\eps)\int_{M}\psi^{q-2}u^{m+2}|\Delta \psi|^{2}\nonumber\\
&\leq \eps \int_{M}\psi^{q}u^{m-2}|\nabla u|^{4}+C(\eps)\int_{U}\psi^{q-4}u^{m+2}|\nabla \psi|^{4}+C(\eps)\int_{U}\psi^{q-4}u^{m+2}|\Delta \psi|^{2}\nonumber\\
&= \eps \int_{M}\psi^{q}u^{m-2}|\nabla u|^{4}+C(\eps)\int_{M}\psi^{q-4}u^{m+2}\bigl(|\nabla \psi|^{4}+|\Delta \psi|^{2}\bigr),
\end{align}
since $0\leq\psi\leq 1$ everywhere. Substituting the above inequalities for $\eps$ small enough into \eqref{eq-ciao}, we obtain
\begin{align}\label{eq-est2}
&\int_{M}\psi^q u^{m-2}|\nabla u|^{4} + \int_M \psi^q Ku^{m+p-1}|\nabla u|^2 \\
&\quad\leq C \int_{M}\psi^{q-2}Ku^{m+p+1}\bigl(|\nabla \psi|^{2}+|\Delta\psi|\bigr)+ C \int_{M}\psi^{q-4}u^{m+2}\bigl(|\nabla \psi|^{4}+|\Delta \psi|^{2}\bigr)
\end{align}
for some positive constant $C$. Multiplying equation \eqref{eq-eq} by $Ku^{m+p}\psi^q$ and integrating by parts, we get
\begin{align*}
\int_M \psi^qK^2 u^{2p+m} &= I_1 + (m+p)\int_M \psi^q Ku^{m+p-1}|\nabla u|^2+\int_M\psi^qu^{p+m}\langle\nabla K,\nabla u\rangle\\
&= I_1 + (m+p)\int_M \psi^q Ku^{m+p-1}|\nabla u|^2+\frac{1}{p+m+1}\int_M\psi^q\langle\nabla K,\nabla u^{p+m+1}\rangle\\
&= I_1 + (m+p)\int_M \psi^q Ku^{m+p-1}|\nabla u|^2-\frac{1}{p+m+1}\int_M\psi^q\Delta K\, u^{p+m+1}\\
&\,\,\,\,\,\,\,-\frac{1}{p+m+1}\int_M\langle\nabla\psi^q,\nabla K\rangle u^{p+m+1}\\
&= 2I_1 + (m+p)\int_M \psi^q Ku^{m+p-1}|\nabla u|^2-\frac{1}{p+m+1}\int_M\psi^q\Delta K\, u^{p+m+1}\\
&\,\,\,\,\,\,\,+\frac{1}{p+m+1}\int_M\Delta\psi^q\, K u^{p+m+1}
\end{align*}
By \eqref{eq-I1} and \eqref{eq-est2} and since $\Delta K\geq0$ we have for some $C>0$
$$
\int_M \psi^q K^2u^{2p+m} \leq C \int_{M}\psi^{q-2}Ku^{m+p+1}\bigl(|\nabla \psi|^{2}+|\Delta\psi|\bigr)+ C \int_{M}\psi^{q-4}u^{m+2}\bigl(|\nabla \psi|^{4}+|\Delta \psi|^{2}\bigr).
$$
We explicitly note that by our choice of $m$ we have $2p+m>p+m+1>0$. By Young's inequality, if $q\geq \frac{2(2p+m)}{p-1}$, we have
\begin{equation*}
\int_M \psi^q K^2u^{2p+m} \leq  C\int_M \psi^{q-\frac{2(2p+m)}{p-1}}K^{-\frac{m+2}{p-1}}\left(|\nabla\psi|^2+|\Delta\psi|\right)^{\frac{2p+m}{p-1}}.
\end{equation*}
Since $(M^n,g)$ has nonnegative Ricci curvature, it is possible to construct cutoff functions such that $\psi\equiv 1$ in $B_\rho$, $\psi\equiv 0$ in $B_{2\rho}^c$, $0\leq \psi\leq 1$ on $M$ and satisfying (see~[Lemma 1.5]\cite{wanzhu} and \cite{checol}, for instance)
$$
|\Delta \psi|+|\nabla \psi|^{2} \leq C \rho^{-2}\quad\text{in }B_{2\rho}\setminus B_{\rho}.
$$
Therefore we deduce
\begin{equation}\label{eq-fin}
\int_{B_\rho} K^2u^{2p+m} \leq  C \rho^{-\frac{2(2p+m)}{p-1}}\int_{B_{2\rho}\setminus B_\rho}K^{-\frac{m+2}{p-1}}.
\end{equation}
If $n=3$, by our choice of $m=-2$ and by Bishop-Gromov volume comparison theorem we obtain
$$
\int_{B_\rho} K^2u^{2(p-1)} \leq  C \rho^{-4}\operatorname{Vol}(B_{2\rho})\leq\frac{C}{\rho}
$$
and thus, passing to the limit as $\rho$ tends to $\infty$, we see that $u=0$ on the set $\{x\in M\,|\,K(x)>0\}$, a contradiction since we assumed $u>0$ on $M$.

If $n\geq 4$, from \eqref{eq-fin}, our assumption on $K$ and Bishop-Gromov volume comparison theorem we have
$$\int_{B_\rho} K^2u^{2p+m} \leq  C \rho^{n-\frac{2(2p+m)}{p-1}+\sigma\frac{m+2}{p-1}}=C\rho^\gamma,$$
and if $\gamma<0$, passing to the limit as $\rho$ tends to $\infty$, we reach a contradiction with our assumption $u>0$ on $M$, as in the previous case.
Now note that we have $\gamma<0$ if and only if $$\sigma<\frac{(-(n-4)p+2m+n)}{m+2},$$ since $m>-2$. Since this expression is increasing in $m$, the optimal choice is $$\sigma<\frac{(-(n-4)p+2m+n)}{m+2},\quad\text{ with }m=\frac{-d+\sqrt{dn(2(n-1)-d(n-2))}}{n-1}.$$
The resulting expression is decreasing in $d$, and hence the optimal choice is
\begin{equation}\label{eq-sigma}
\sigma<\sigma^*=\frac{2n(p+1)-n^2(p-1)+4p+4\sqrt{np(2(1+p)-n(p-1))}}{2(2+n-p+\sqrt{np(2(1+p)-n(p-1))})}.
\end{equation}
A simple computation shows that $\sigma^*\geq\frac{2}{n-3}$ and hence the result follows. We explicitly note that $\sigma^*=\frac{2}{n-3}$ if $n=4$ or if $p=p_c=\frac{n+2}{n-2}$.

\end{proof}

\

\section{Proof of Theorem \ref{t-cri}}\label{s-cri}

\begin{proof} Let $(M^n,g)$, $n\geq 3$, be a Riemannian manifold with nonnegative Ricci curvature and let $u\in C^{2}(M)$ be a nonnegative solution of \eqref{eq-eq} with $p=p_c=\frac{n+2}{n-2}$. We assume $u$ is nontrivial, and hence by the strong Maximum principle $u$ is positive on $M$.
We now use Corollary \ref{cor1} and Lemma \ref{lem1} with $q=2$ and $\varepsilon=\frac{1}{2}$ and we obtain that for any nonnegative $\psi \in C^{2}_{c}(M)$
\begin{align*}
&\frac{(n-2)^2}{8}\int_M \psi^2 u^\frac{2(n-1)}{n-2} \left| \mathring{\nabla}^2 u^{-\frac{2}{n-2}}\right|^2+\frac{(n-2)^2}{4}\int_M \psi^2 u^\frac{2(n-1)}{n-2} \operatorname{Ric}\left(\nabla u^{-\frac{2}{n-2}}, \nabla u^{-\frac{2}{n-2}}\right)\\
&\qquad\,\,\,+\frac{n-2}{2n}\int_M\psi^2\Delta K\,u^{\frac{2(n-1)}{n-2}}\\
&\qquad\leq 2\int_{M}u^{-\frac{2}{n-2}}|\nabla u|^{2}|\nabla \psi|^2+\frac{n-1}{n}\int_MKu^\frac{n}{n-2}\langle\nabla u,\nabla\psi^2\rangle+\frac{n-2}{2n}\int_M\Delta\psi^2\,Ku^{\frac{2(n-1)}{n-2}}.
\end{align*}
For any $\rho>1$, let $\psi\in C^2_c(M)$ be such that $\psi\equiv 1$ in $B_\rho$, $\psi\equiv 0$ in $B_{2\rho}^c$, $0\leq \psi\leq 1$ on $M$ and $\psi$ satisfies
$$
|\nabla \psi|^{2}+|\Delta\psi| \leq C \rho^{-2}\quad\text{in }B_{2\rho}\setminus B_{\rho}.
$$
Then, using Lemma \ref{lem2}, for some constant $C>0$ we have
$$\int_{M}u^{-\frac{2}{n-2}}|\nabla u|^{2}|\nabla \psi|^2\leq C\int_{B_\rho^c}|\nabla u|^2$$
If $0\leq K\leq C(1+r^2)$ outside a compact set of $M$, using Lemma \ref{lem2} we have
$$\int_M|\Delta\psi^2|\,Ku^{\frac{2(n-1)}{n-2}}\leq\frac{C}{\rho^2}\int_{B_{2\rho}\setminus B_\rho}Ku^{\frac{2n}{n-2}}u^{-\frac{2}{n-2}}\leq C\int_{B_\rho^c}Ku^{\frac{2n}{n-2}}$$
and
\begin{align*}
\int_MKu^\frac{n}{n-2}|\nabla u||\nabla\psi^2|&\leq C\left(\int_{B_{2\rho}\setminus B_\rho}K^2u^\frac{2n}{n-2}|\nabla\psi|^2\right)^\frac{1}{2}\left(\int_{B_\rho^c}|\nabla u|^2\right)^\frac{1}{2}\\
&\leq C\left(\int_{B_{2\rho}\setminus B_\rho}Ku^\frac{2n}{n-2}\right)^\frac{1}{2}\left(\int_{B_\rho^c}|\nabla u|^2\right)^\frac{1}{2}.
\end{align*}
On the other hand, if $|\nabla K(x)|\leq \frac{C}{r(x)}K(x)$ outside a compact set of $M$, we have
\begin{align*}
\frac{n-1}{n}\int_MKu^\frac{n}{n-2}\langle\nabla u,\nabla\psi^2\rangle&=\frac{n-2}{2n}\int_MK\langle\nabla u^{\frac{2(n-1)}{n-2}},\nabla\psi^2\rangle\\
&= -\frac{n-2}{2n}\int_M\Delta\psi^2 \,Ku^{\frac{2(n-1)}{n-2}}-\frac{n-2}{2n}\int_Mu^{\frac{2(n-1)}{n-2}}\langle\nabla K,\nabla\psi^2\rangle
\end{align*}
and using Lemma \ref{lem2}, for $\rho>1$ large enough, we obtain
\begin{align*}
&\left|\frac{n-1}{n}\int_MKu^\frac{n}{n-2}\langle\nabla u,\nabla\psi^2\rangle+\frac{n-2}{2n}\int_M\Delta\psi^2\,Ku^{\frac{2(n-1)}{n-2}}\right|\\
&\qquad\leq
C\int_{B_{2\rho}\setminus B_\rho}|\nabla K| |\nabla\psi| u^{\frac{2n}{n-2}} u^{-\frac{2}{n-2}}\leq C \int_{B^c_\rho}Ku^\frac{2n}{n-2}.
\end{align*}
In either case, since $\Delta K\geq0$ on $M$, for every $\rho>1$ large enough we have
$$
\int_{B_\rho} u^\frac{2(n-1)}{n-2} \left| \mathring{\nabla}^2 u^{-\frac{2}{n-2}}\right|^2+\int_{B_\rho} u^\frac{2(n-1)}{n-2} \operatorname{Ric}\left(\nabla u^{-\frac{2}{n-2}}, \nabla u^{-\frac{2}{n-2}}\right)\leq C\left(\int_{B_{\rho}^c}|\nabla u|^{2}+\int_{B^c_\rho}Ku^\frac{2n}{n-2}\right)
$$
for some constant $C>0$. Passing to the limit as $\rho$ tends to $\infty$, since $\operatorname{Ric}\geq0$ we conclude that
$$\mathring{\nabla}^2 u^{-\frac{2}{n-2}}\equiv0,\qquad \operatorname{Ric}\left(\nabla u^{-\frac{2}{n-2}}, \nabla u^{-\frac{2}{n-2}}\right)\equiv0\qquad\text{ on }M.$$

From Lemma \ref{lem5} we deduce that $(M^n,g)$ is either
\begin{itemize}
     \item[i)] a direct product $(\mathbb{R}\times N^{n-1}, dr^2+g_N)$, where $(N^{n-1},g_N)$ is a $(n-1)$-dimensional complete Riemannian manifold with nonnegative Ricci curvature and $u^{-\frac{2}{n-2}}=f=f(r)=ar+b$ for some $a,b\in\mathbb{R}$ with $a\neq0$, or
     \item[ii)] $\mathbb{R}^n$ with the Euclidean metric and $u^{-\frac{2}{n-2}}=f=f(x)=a|x-x_0|^2+b$ for some $a,b\in\mathbb{R}$ with $a\neq0$.
   \end{itemize}
Since
$$\nabla f=-\frac{2}{n-2}u^{-\frac{n}{n-2}}\nabla u,\qquad \Delta f=-\frac{2}{n-2}u^{-\frac{n}{n-2}}\Delta u+\frac{2n}{(n-2)^2}u^{-\frac{2(n-1)}{n-2}}|\nabla u|^2$$
a simple computation shows that in the first case
$$0=\Delta f=\frac{2}{n-2}Ku^\frac{2}{n-2}+\frac{2n}{(n-2)^2}u^{-\frac{2(n-1)}{n-2}}|\nabla u|^2\geq0\qquad\text{ on }M,$$
which is in contradiction with our assumption $u>0$ on $M$. Then $(M^n,g)$ is $\mathbb{R}^n$ with the Euclidean metric and $u^{-\frac{2}{n-2}}=f=f(x)=a|x-x_0|^2+b$ for some $x_0\in\mathbb{R}^n$, $a,b\in\mathbb{R}$ with $a\neq0$. Since $u,f$ are positive functions we must have $a,b>0$ and thus $$u(x)=(a|x-x_0|^2+b)^{-\frac{n-2}{2}}.$$
Inserting this expression into the equation we find that we must have $K\equiv n(n-2)ab\in(0,\infty)$, and thus we conclude.
\end{proof}

\

\section{Proof of Theorem \ref{t-cri2}}\label{s-cri2}

\begin{proof} Let $(M^n,g)$, $n\geq 3$, be a Riemannian manifold with nonnegative Ricci curvature and let $u\in C^{2}(M)$ be a nonnegative solution of \eqref{eq-eq} with $p=p_c=\frac{n+2}{n-2}$, $K\in C^2(M)$, $0\not\equiv K\geq0$ and $\Delta K\geq0$. We assume $u$ is nontrivial, and hence by the strong maximum principle $u$ is positive on $M$.
We now use Corollary \ref{cor1} and Lemma \ref{lem1} with $\varepsilon=\frac{1}{q}$ and we obtain that for any nonnegative $\psi \in C^{2}_{c}(M)$
\begin{align*}
&\frac{(n-2)^2}{8}\int_M \psi^q u^\frac{2(n-1)}{n-2} \left| \mathring{\nabla}^2 u^{-\frac{2}{n-2}}\right|^2+\frac{(n-2)^2}{4}\int_M \psi^q u^\frac{2(n-1)}{n-2} \operatorname{Ric}\left(\nabla u^{-\frac{2}{n-2}}, \nabla u^{-\frac{2}{n-2}}\right)\\
&\qquad\,\,\,+\frac{n-2}{2n}\int_M\psi^q\Delta K\,u^\frac{2(n-1)}{(n-2)}\\
&\quad\leq \frac{q^2}{2}\int_{M}\psi^{q-2}u^{-\frac{2}{n-2}}|\nabla u|^{2}|\nabla \psi|^2 +\frac{n-1}{n}\int_MKu^{\frac{n}{n-2}}\langle\nabla\psi^q,\nabla u\rangle
+\frac{n-2}{2n}\int_M\Delta \psi^q\,Ku^\frac{2(n-1)}{(n-2)}.
\end{align*}
We have
\begin{align*}
  &\frac{n-1}{n}\int_MKu^{\frac{n}{n-2}}\langle\nabla\psi^q,\nabla u\rangle=\frac{n-2}{2n}\int_MK\langle\nabla\psi^q,\nabla u^{\frac{2(n-1)}{(n-2)}}\rangle\\
&\qquad=-\frac{n-2}{2n}\int_MK\Delta\psi^q\,u^\frac{2(n-1)}{(n-2)}-\frac{n-2}{2n}\int_M\langle\nabla\psi^q,\nabla K\rangle u^\frac{2(n-1)}{(n-2)}.
\end{align*}
Hence, by our assumptions on $K$ and since $0\leq\psi\leq1$,
\begin{align*}
&\left|\frac{n-1}{n}\int_MKu^\frac{n}{n-2}\langle\nabla u,\nabla\psi^q\rangle+\frac{n-2}{2n}\int_M\Delta\psi^q\,Ku^{\frac{2(n-1)}{n-2}}\right|\\
&\qquad\leq
C\int_{B_{2\rho}\setminus B_\rho}|\nabla K| \psi^{q-1}|\nabla\psi| u^{\frac{2(n-1)}{n-2}} \leq \frac{\bar{C}}{\rho^2} \int_{B_{2\rho}\setminus B_\rho}\psi^{q-2}Ku^\frac{2(n-1)}{n-2},
\end{align*}
and thus
\begin{align*}
&\frac{(n-2)^2}{8}\int_M \psi^q u^\frac{2(n-1)}{n-2} \left| \mathring{\nabla}^2 u^{-\frac{2}{n-2}}\right|^2+\frac{(n-2)^2}{4}\int_M \psi^q u^\frac{2(n-1)}{n-2} \operatorname{Ric}\left(\nabla u^{-\frac{2}{n-2}}, \nabla u^{-\frac{2}{n-2}}\right)\\
&\quad\leq 8\int_{M}\psi^{q-2}u^{-\frac{2+\gamma}{n-2}}|\nabla u|^{2}u^\frac{\gamma}{n-2}|\nabla \psi|^2 +\frac{\bar{C}}{\rho^2} \int_{B_{2\rho}\setminus B_\rho}\psi^{q-2}Ku^\frac{2(n-1)}{n-2}
\end{align*}
for any $\gamma\geq 0$. Multiplying the equation for $u$ by $\psi^{q-2} u^{\frac{n-4-\gamma}{n-2}}$ and integrating by parts we obtain
\begin{equation}\label{eq-2}
\frac{n-4-\gamma}{n-2}\int_M \psi^{q-2} u^{-\frac{2+\gamma}{n-2}}|\nabla u|^{2} +(q-2)\int_M \psi^{q-3} \,u^{\frac{n-4-\gamma}{n-2}}\langle\nabla u,\nabla \psi\rangle = \int_M \psi^{q-2} K u^{\frac{2n-2-\gamma}{n-2}}.
\end{equation}
Hence, if $n=3$ we choose $\gamma=0$, $q=4$ thus obtaining
\begin{align*}
\int_{M}\psi^2u^{-2}|\nabla u|^{2} &= 2\int_M \psi \,u^{-1}\langle\nabla u,\nabla \psi\rangle -\int_M\psi^2Ku^4 \\
&\leq \frac12 \int_{M}\psi^2u^{-2}|\nabla u|^{2} + 2 \int_M |\nabla \psi|^2-\int_M\psi^2Ku^4.
\end{align*}
For any $\rho>1$, let $\psi\in C^2_c(M)$ be such that $\psi\equiv 1$ in $B_\rho$, $\psi\equiv 0$ in $B_{2\rho}^c$, $0\leq \psi\leq 1$ on $M$ and $\psi$ satisfies
$$
|\nabla \psi|^{2} \leq C \rho^{-2} \quad\text{in }B_{2\rho}\setminus B_{\rho}.
$$
By Bishop-Gromov volume estimate, we conclude
$$
\int_{B_{2\rho}}\psi^2u^{-2}|\nabla u|^{2} \leq C \rho-2\int_M\psi^2Ku^4\leq C_1\left(C\rho-2\int_M\psi^2Ku^4\right),
$$
for every $C_1\geq1$, and thus
\begin{align*}
&\int_M \psi^4 u^4 \left| \mathring{\nabla}^2 u^{-2}\right|^2+\int_M \psi^4 u^4 \operatorname{Ric}\left(\nabla u^{-2}, \nabla u^{-2}\right)\\
&\qquad\leq C\int_{M}\psi^2u^{-2}|\nabla u|^{2}|\nabla \psi|^2+\frac{\bar{C}}{\rho^2} \int_{B_{2\rho}\setminus B_\rho}\psi^2Ku^4\\
&\qquad\leq \frac{CC_1}{\rho^2}\left(C\rho-2\int_M\psi^2Ku^4\right)+\frac{\bar{C}}{\rho^2} \int_{B_{2\rho}\setminus B_\rho}\psi^2Ku^4\leq\frac{C}{\rho},
\end{align*}
if $C_1\geq1$ is chosen large enough. Letting $\rho\to\infty$ we deduce that
$$\mathring{\nabla}^2 u^{-2}\equiv0,\qquad \operatorname{Ric}\left(\nabla u^{-2}, \nabla u^{-2}\right)\equiv0\qquad\text{ on }M,$$
and we conclude as in the proof of Theorem \ref{t-cri}.
If $n\geq 4$, for every $n-4<\gamma<2n-6$ and every $\eps>0$ by \eqref{eq-2} we have
\begin{align*}
  \int_M \psi^{q-2} u^{-\frac{2+\gamma}{n-2}}|\nabla u|^{2}&\leq (q-2)C\int_M \psi^{q-3} \,u^{\frac{n-4-\gamma}{n-2}}\langle\nabla u,\nabla \psi\rangle-C\int_M \psi^{q-2} K u^{\frac{2n-2-\gamma}{n-2}}  \\
   & \leq C\eps\int_M \psi^{q-2} \,u^{-\frac{2+\gamma}{n-2}}|\nabla u|^2+\frac{C}{\eps}\int_M \psi^{q-4} \,u^{\frac{2n-6-\gamma}{n-2}}|\nabla \psi|^2\\
   &\quad-C\int_M \psi^{q-2} K u^{\frac{2n-2-\gamma}{n-2}}
\end{align*}
and thus for some $C>0$ depending on $\gamma$ we have
\begin{align}\label{eq-3}
  \int_M \psi^{q-2} u^{-\frac{2+\gamma}{n-2}}|\nabla u|^{2}&\leq C\int_M \psi^{q-4} \,u^{\frac{2n-6-\gamma}{n-2}}|\nabla \psi|^2-C\int_M \psi^{q-2} K u^{\frac{2n-2-\gamma}{n-2}}.
\end{align}
Now for every $\eps>0$, since $\gamma<2n-6$, we have
\begin{align*}
  \int_M \psi^{q-4} \,u^{\frac{2n-6-\gamma}{n-2}}|\nabla \psi|^2\leq \frac{\eps}{2} \int_M \psi^{q-2} Ku^{\frac{2n-2-\gamma}{n-2}}+\frac{1}{2\eps}\int_M\psi^{q-1-n+\frac{\gamma}{2}} K^{-\frac{2n-6-\gamma}{4}}|\nabla\psi|^{\frac{2n-2-\gamma}{2}}.
\end{align*}
Choosing $\eps>0$ suitably small and substituting in \eqref{eq-3} we find
\begin{align*}
  \int_M \psi^{q-2} u^{-\frac{2+\gamma}{n-2}}|\nabla u|^{2}&\leq C\int_M\psi^{q-1-n+\frac{\gamma}{2}} |\nabla\psi|^{\frac{2n-2-\gamma}{2}}K^{-\frac{2n-6-\gamma}{4}}-\int_M\psi^{q-2}Ku^{\frac{2n-2-\gamma}{n-2}}\\
  &\leq C_1\left(C\int_M\psi^{q-1-n+\frac{\gamma}{2}} |\nabla\psi|^{\frac{2n-2-\gamma}{2}}K^{-\frac{2n-6-\gamma}{4}}-\int_M\psi^{q-2}Ku^{\frac{2n-2-\gamma}{n-2}}\right),
\end{align*}
for large enough $C>0$ and every $C_1\geq1$. For any $\rho>1$, we choose again $\psi\in C^2_c(M)$ such that $\psi\equiv 1$ in $B_\rho$, $\psi\equiv 0$ in $B_{2\rho}^c$, $0\leq \psi\leq 1$ on $M$ and such that $\psi$ satisfies
$$
|\nabla \psi|^{2} \leq C \rho^{-2}\quad\text{in }B_{2\rho}\setminus B_{\rho}.
$$
Thus, by our assumptions on $u,K$ and using the Bishop-Gromov volume comparison theorem, choosing $q\geq1+n-\frac{\gamma}{2}$ and $C_1\geq1$ large enough we obtain
\begin{align*}
&\int_M \psi^q u^\frac{2(n-1)}{n-2} \left| \mathring{\nabla}^2 u^{-\frac{2}{n-2}}\right|^2+\int_M \psi^q u^\frac{2(n-1)}{n-2} \operatorname{Ric}\left(\nabla u^{-\frac{2}{n-2}}, \nabla u^{-\frac{2}{n-2}}\right)\\
&\qquad\leq C(\gamma)\int_{M}\psi^{q-2}u^{-\frac{2+\gamma}{n-2}}|\nabla u|^{2} u^{\frac{\gamma}{n-2}}|\nabla \psi|^2+\frac{\bar{C}}{\rho^2} \int_{B_{2\rho}\setminus B_\rho}\psi^{q-2}Ku^\frac{2(n-1)}{n-2}\\
&\qquad\leq \frac{1}{\rho^2}\left(\sup_{B_{2\rho}\setminus B_\rho}u\right)^\frac{\gamma}{n-2}\left(C\int_{M}\psi^{q-2}u^{-\frac{2+\gamma}{n-2}}|\nabla u|^{2}+\bar{C} \int_{B_{2\rho}\setminus B_\rho}\psi^{q-2}Ku^\frac{2n-2-\gamma}{n-2}\right)\\
&\qquad\leq \frac{1}{\rho^2}\left(\sup_{B_{2\rho}\setminus B_\rho}u\right)^\frac{\gamma}{n-2}\bigg(CC_1\int_M\psi^{q-1-n+\frac{\gamma}{2}} |\nabla\psi|^{\frac{2n-2-\gamma}{2}}K^{-\frac{2n-6-\gamma}{4}}-C_1\int_M\psi^{q-2}Ku^{\frac{2n-2-\gamma}{n-2}}\\
&\qquad\qquad+\bar{C} \int_{B_{2\rho}\setminus B_\rho}\psi^{q-2}Ku^\frac{2n-2-\gamma}{n-2}\bigg)\\
&\qquad\leq \frac{C}{\rho^2}\left(\sup_{B_{2\rho}\setminus B_\rho}u\right)^\frac{\gamma}{n-2}\int_M |\nabla\psi|^{\frac{2n-2-\gamma}{2}}K^{-\frac{2n-6-\gamma}{4}}
\leq C\rho^{\gamma\left(\frac{\alpha}{n-2}+\frac{1}{2}-\frac{\sigma}{4}\right)+\frac{n-3}{2}\sigma-1}.
\end{align*}
Now if $$\lambda=\lambda(\gamma):=\gamma\left(\frac{\alpha}{n-2}+\frac{1}{2}-\frac{\sigma}{4}\right)+\frac{n-3}{2}\sigma-1<0$$
for some $\gamma\in(n-4,2n-6)$, passing to the limit as $\rho$ tends to $\infty$ yields
$$\mathring{\nabla}^2 u^{-\frac{2}{n-2}}\equiv0,\qquad \operatorname{Ric}\left(\nabla u^{-\frac{2}{n-2}}, \nabla u^{-\frac{2}{n-2}}\right)\equiv0\qquad\text{ on }M,$$
and again we can conclude as in the proof of Theorem \ref{t-cri}. Since $\lambda(\gamma)$ is monotone in $\gamma$, we have that $\lambda(\gamma)<0$ for some $\gamma\in(n-4,2n-6)$ if and only if $\lambda(n-4)<0$ or $\lambda(2n-6)<0$. These conditions are equivalent to, respectively, $$\alpha<-\frac{(n-2)^2\sigma+2(n-2)(n-6)}{4(n-4)}\quad\text{or}\quad\alpha<-\frac{(n-2)(n-4)}{2(n-3)}.$$ Then by our assumptions on $\alpha,\sigma$ there exists $\gamma\in(n-4,2n-6)$ such that $\lambda(\gamma)<0$, and the proof is complete.
\end{proof}

\

\section{Proof of Theorem \ref{t-cri3}}\label{s-cri3}

By Corollary \ref{cor2} for every nonnegative $\psi \in C^{2}_{c}(M)$ and every $q\geq2$, since $\Delta K\geq 0$, we have
\begin{align}\label{eq-iniz}
8\int_M \psi^q e^{\frac{u}{2}} \left| \mathring{\nabla}^2 e^{-\frac{u}{2}}\right|^2+4\int_M \psi^q e^{\frac{u}{2}} R \left|\nabla e^{-\frac{u}{2}}\right|^2\leq 2 \int_{M} K e^\frac{u}{2}\Delta\psi^q+\int_M e^{-\frac{u}{2}}|\nabla u|^{2}\Delta\psi^q.
\end{align}
Integrating by parts and using Young's inequality, for every $\eps>0$, we have
\begin{align*}
\frac 14 &\int_M e^{-\frac{u}{2}}|\nabla u|^{2}\Delta\psi^q =\int_M e^{\frac{u}{2}} \left|\nabla e^{-\frac{u}{2}}\right|^2\Delta\psi^q \\
&= -\frac12 \int_M e^{\frac{u}{2}}\left|\nabla e^{-\frac{u}{2}}\right|^2 \langle \nabla u,\nabla \psi^q\rangle -\int_M e^{\frac{u}{2}} \langle \nabla \left|\nabla e^{-\frac{u}{2}}\right|^2, \nabla \psi^q \rangle\\
&=-\frac18 \int_M e^{-\frac{u}{2}}|\nabla u|^2 \langle \nabla u,\nabla \psi^q\rangle -2\int_M e^{\frac{u}{2}}  \nabla^2 e^{-\frac{u}{2}}\left(\nabla e^{-\frac{u}{2}}, \nabla \psi^q \right)\\
&= -\frac18 \int_M e^{-\frac{u}{2}}|\nabla u|^2 \langle \nabla u,\nabla \psi^q\rangle -2\int_M e^{\frac{u}{2}}  \mathring{\nabla}^2 e^{-\frac{u}{2}}\left(\nabla e^{-\frac{u}{2}}, \nabla \psi^q \right)-\int_M e^{\frac{u}{2}}  \Delta e^{-\frac{u}{2}}\langle\nabla e^{-\frac{u}{2}}, \nabla \psi^q \rangle\\
&\leq -\frac18 \int_M e^{-\frac{u}{2}}|\nabla u|^2 \langle \nabla u,\nabla \psi^q\rangle +\eps\int_M \psi^q e^{\frac{u}{2}}  \left|\mathring{\nabla}^2 e^{-\frac{u}{2}}\right|^2+\frac{C}{\eps}\int_M \psi^{q-2} e^{-\frac{u}{2}}|\nabla u|^2 |\nabla\psi|^2 \\
&\quad +\frac18 \int_M e^{-\frac{u}{2}}|\nabla u|^2 \langle \nabla u,\nabla \psi^q\rangle +\frac14\int_M K e^{\frac{u}{2}}\langle\nabla u,\nabla \psi^q\rangle.
\end{align*}
Then, from \eqref{eq-iniz} choosing $\eps=1$, $q=4$, we obtain
\begin{align}\label{eq-iniz2}
4\int_M & \psi^4 e^{\frac{u}{2}} \left| \mathring{\nabla}^2 e^{-\frac{u}{2}}\right|^2+4\int_M \psi^4 e^{\frac{u}{2}} R \left|\nabla e^{-\frac{u}{2}}\right|^2\\\nonumber
&\leq 2 \int_{M} K e^\frac{u}{2}\Delta\psi^4+C\int_M \psi^{2} e^{-\frac{u}{2}}|\nabla u|^2 |\nabla\psi|^2 +\int_M K e^{\frac{u}{2}}\langle\nabla u,\nabla \psi^4\rangle
\end{align}

{\em Case 1:} suppose that $0\leq K(x)\leq C r(x)^2$ outside a compact set of $M$. Since $M$ has nonnegative Ricci curvature, by using \cite{biaset}, one can see that there exist $C>0$ and  $\theta\in(0,\frac 12)$ such that, for any $\rho$ large enough, there exists a cut-off function $\psi\in C^2(M)$ such that $\psi\equiv1$ on $B_{\rho^\theta}$, $\psi\equiv0$ on $B_{\rho}^c$ and
$$|\nabla\psi|\leq\frac{C}{r \log\rho}, \quad |\Delta \psi|\leq\frac{C}{r^2 \log\rho}\qquad\text{on } B_{\rho}\setminus B_{\rho^\theta}.$$
Note also that, up to increasing the constant $C>0$, we have $$|\nabla\psi|\leq\frac{C}{(r+1) \log\rho}, \qquad\text{on } B_{\rho}\setminus B_{\rho^\theta}.$$
With this choice of $\psi$, by our assumption on $u$, for $\rho>1$ large enough
\begin{align}\label{eq-guardone}
\int_M \psi^{2} e^{-\frac{u}{2}}|\nabla u|^2 |\nabla\psi|^2 &\leq \frac{C}{(\log\rho)^2}\int_M\frac{\psi^{2}}{(r+1)^2}e^{-\frac{u}{2}}|\nabla u|^2
\end{align}
and
\begin{align*}
\int_M\frac{\psi^{2}}{(r+1)^2}e^{-\frac{u}{2}}|\nabla u|^2 &=-2\int_M\frac{\psi^{2}}{(r+1)^2}\langle \nabla e^{-\frac{u}{2}},\nabla u\rangle\\
&= 2\int_M\frac{1}{(r+1)^2} e^{-\frac{u}{2}}\langle \nabla \psi^2 ,\nabla u\rangle
+2\int_M\frac{\psi^2}{(r+1)^2} e^{-\frac{u}{2}}\Delta u\\
&\quad -4\int_M\frac{\psi^2}{(r+1)^3} e^{-\frac{u}{2}}\langle \nabla r ,\nabla u\rangle\\
&\leq \frac14 \int_M\frac{\psi^{2}}{(r+1)^2}e^{-\frac{u}{2}}|\nabla u|^2 + C \int_M \frac{1}{(r+1)^2}e^{-\frac u2}|\nabla\psi|^2-2\int_M \frac{\psi^2}{(r+1)^2}Ke^{\frac u2}\\
&\quad+\frac14  \int_M\frac{\psi^{2}}{(r+1)^2}e^{-\frac{u}{2}}|\nabla u|^2 + C \int_M \frac{\psi^2}{(r+1)^4}e^{-\frac u2}.
\end{align*}
Thus, by our assumptions on $u$, the choice of $\psi$ and Bishop-Gromov volume comparison, arguing as in the proof of Theorem \ref{t-sub} when $n=2$, we have
\begin{align*}
\int_M\frac{\psi^{2}}{(r+1)^2}e^{-\frac{u}{2}}|\nabla u|^2 &\leq C \int_M \frac{1}{(r+1)^2}e^{-\frac u2}|\nabla\psi|^2+ C \int_M \frac{\psi^2}{(r+1)^4}e^{-\frac u2}\\
&\leq \frac{C}{(\log\rho)^2} \int_{B_\rho\setminus B_{\rho^\theta}} \frac {(\log r)^\gamma}{r^2} +C \int_{B_\rho}  \frac {r^2(\log r)^\gamma}{(r+1)^4}\\
&\leq C (\log\rho)^{1+\gamma}.
\end{align*}
From \eqref{eq-guardone} for every $\rho>1$ large enough we obtain
\begin{equation}\label{eq-c2}
\int_M \psi^{2} e^{-\frac{u}{2}}|\nabla u|^2 |\nabla\psi|^2\leq \frac{C}{ (\log\rho)^{1-\gamma}}.
\end{equation}
Moreover
$$
 \int_{M} K e^\frac{u}{2}\Delta\psi^4 \leq C \int_{B_\rho\setminus B_{\rho^{\theta}}} K e^u \frac{e^{-\frac u2}}{r^2\log\rho}\leq C\int_{(B_{\rho^{\theta}})^c} K e^u.
$$
Finally
\begin{align*}
\int_M K e^{\frac{u}{2}}\langle\nabla u,\nabla \psi^4\rangle &\leq \frac{1}{2}\int_{(B_{\rho^\theta})^c} K e^ u +\frac{1}{2}\int_{B_{\rho}\setminus B_{\rho^\theta}} K |\nabla u|^2 |\nabla\psi|^2\psi^2\\
&\leq \frac{1}{2}\int_{(B_{\rho^{\theta}})^c} K e^ u +\frac{C}{(\log\rho)^2}\int_{B_{\rho}} |\nabla u|^2 \psi^2.
\end{align*}
We multiply the equation satisfied by $u$, i.e. \eqref{eq-eq2},  by $\psi^2u$ and integrate by parts. Since $K\geq0$, $0\leq\psi\leq1$ on $M$, $Ke^u\in L^1(M)$ and $-4\log r -\gamma\log\log r\leq u\leq C \log r $ for $r$ large enough, for every large enough $\rho>1$ we have
\begin{align*}
\int_{B_{\rho}} |\nabla u|^2 \psi^2&=\int_M\psi^2Kue^u-2\int_M\psi u\langle\nabla u,\nabla\psi\rangle\\
&\leq \int_{B_{\rho}}\psi^2K u e^u +\frac12\int_{B_{\rho}} |\nabla u|^2 \psi^2+ 2\int_{B_{\rho}\setminus B_{\rho^\theta}} u^2|\nabla \psi|^2\\
&\leq \frac12\int_{B_{\rho}} |\nabla u|^2\psi^2+C\log\rho +C,
\end{align*}
where we also used Bishop-Gromov volume comparison theorem. Then
$$
\int_{B_{\rho}} |\nabla u|^2 \psi^2 \leq C\log\rho +C
$$
and therefore
$$
\int_M K e^{\frac{u}{2}}\langle\nabla u,\nabla \psi^4\rangle  \leq \frac{1}{2}\int_{(B_{\rho^{\theta}})^c} K e^ u+\frac{C}{\log\rho}.
$$
Inserting these estimates into \eqref{eq-iniz2}, we obtain for every $\rho>1$ large enough
\begin{align*}
4\int_M  \psi^4 e^{\frac{u}{2}} \left| \mathring{\nabla}^2 e^{-\frac{u}{2}}\right|^2+4\int_M \psi^4 e^{\frac{u}{2}} R \left|\nabla e^{-\frac{u}{2}}\right|^2\leq C\int_{(B_{\rho^{\theta}})^c} K e^ u+ \frac{C}{ (\log\rho)}+\frac{C}{ (\log\rho)^{1-\gamma}}.
\end{align*}
Passing to the limit as $\rho$ tends to $\infty$  we obtain
$$\mathring{\nabla}^2 e^{-\frac{u}{2}}\equiv0,\qquad R\left|\nabla e^{-\frac{u}{2}}\right|^2\equiv0\qquad\text{ on }M,$$
and again we can conclude as in the proof of Theorem \ref{t-cri}.

\medskip

{\em Case 2:} if outside a compact set $|\nabla K(x)|\leq \frac{C}{r(x)}K(x)$, integrating by parts we have
\begin{align*}
\int_M K e^{\frac{u}{2}}\langle\nabla u,\nabla \psi^4\rangle &= 2\int_M K \langle \nabla e^{\frac u2},\nabla \psi^4\rangle \\
&= -2\int_M e^{\frac{u}{2}} \langle \nabla K,\nabla\psi^4\rangle - 2 \int_M K e^{\frac{u}{2}} \Delta \psi^4\\
&\leq C\int_M \frac{K}{r}e^{\frac{u}{2}}|\nabla\psi| - 2 \int_M K e^{\frac{u}{2}} \Delta \psi^4.
\end{align*}
Substituting in \eqref{eq-iniz2} we obtain
\begin{align}\label{eq-iniz3}
4\int_M & \psi^4 e^{\frac{u}{2}} \left| \mathring{\nabla}^2 e^{-\frac{u}{2}}\right|^2+4\int_M \psi^4 e^{\frac{u}{2}} R \left|\nabla e^{-\frac{u}{2}}\right|^2\\\nonumber
&\leq C\int_M \psi^{2} e^{-\frac{u}{2}}|\nabla u|^2 |\nabla\psi|^2 +C\int_M \frac{K}{r}e^{\frac{u}{2}}|\nabla\psi|.
\end{align}
Let $\Phi\in C^1(\mathbb{R})$ be a nonnegative function such that $0\leq\Phi\leq1$, $\Phi'\leq0$, $\Phi=1$ on $(-\infty,1]$, $\Phi=0$ on $[2,-\infty)$ and $|\Phi'|\leq C$ for some positive constant $C$. For any $\rho>1$ define $\psi=\Phi\left(\frac{2\log r}{\log \rho}\right)$, then we have $\psi\equiv1$ on $B_{\sqrt\rho}$, $\psi\equiv0$ on $B_{\rho}^c$ and $|\nabla\psi|\leq\frac{C}{r \log\rho}$ on $B_{\rho}\setminus B_{\sqrt\rho}$.
With this choice of $\psi$, by our assumption on $u$, we get
$$
\int_M \frac{K}{r}e^{\frac{u}{2}}|\nabla\psi| \leq \frac{C}{\log\rho}\int_{B_{\rho}\setminus B_{\sqrt\rho}}Ke^u \left( \frac{e^{-\frac u2}}{r^2}\right)\leq C\int_{B_{\sqrt\rho}^c}Ke^u.
$$
On the other hand, arguing as in Case 1, we obtain inequality \eqref{eq-c2}. Finally, using \eqref{eq-iniz3}, we have
$$
4\int_M \psi^4 e^{\frac{u}{2}} \left| \mathring{\nabla}^2 e^{-\frac{u}{2}}\right|^2+4\int_M \psi^4 e^{\frac{u}{2}} R \left|\nabla e^{-\frac{u}{2}}\right|^2 \leq  \frac{C}{ (\log\rho)^{1-\gamma}}+ C\int_{B_{\sqrt\rho}^c}Ke^u.
$$
Passing to the limit as $\rho$ tends to $\infty$  we obtain
$$\mathring{\nabla}^2 e^{-\frac{u}{2}}\equiv0,\qquad R\left|\nabla e^{-\frac{u}{2}}\right|^2\equiv0\qquad\text{ on }M,$$
and again we can conclude as in the proof of Theorem \ref{t-cri}.

\

\

\begin{ackn}
\noindent The first author is member of the {\em GNSAGA, Gruppo Nazionale per le Strutture Algebriche, Geometriche e le loro Applicazioni} of Indam. The second author is a member of {\em GNAMPA, Gruppo Nazionale per l'Analisi Matematica, la Probabilità e le loro Applicazioni} of Indam.
\end{ackn}

\

\

\bibliographystyle{amsplain}

\begin{thebibliography}{999}

\bibitem{aub} T. Aubin, Problèmes isopérimétriques et espaces de Sobolev, J. Differential Geometry 11 (1976), no. 4, 573--598.


\bibitem{biaset} D. Bianchi, A. G. Setti, Laplacian cut-offs, porous and fast diffusion on manifolds and other applications, Calc. Var. Partial Differential Equations 57 (2018), no. 1, Paper No. 4, 33 pp.

\bibitem{bidrao} M.-F. Bidaut-V\'eron, T. Raoux, Asymptotics of solutions of some nonlinear elliptic systems, Comm. in Partial Differential Equations 21 (1996) 7-8, 1035--1086.


\bibitem{bre} S. Brendle, Sobolev inequalities in manifolds with nonnegative curvature, to appear in Comm. Pure Appl. Math.

\bibitem{bremer} H. Brezis, F. Merle, Uniform estimates and blow-up behavior for solutions of $-\Delta u = V(x) e^u$ in two dimensions,
Comm. Partial Differential Equations 16 (1991), no. 8-9, 1223--1253.



\bibitem{cafgidspr} L. Caffarelli, B. Gidas, J. Spruck, Asymptotic symmetry and local behavior of semilinear elliptic equations with critical Sobolev growth, Comm. Pure Appl. Math. 42 (1989), no. 3, 271--297.


\bibitem{car} G. Carron, Inégalités isopérimétriques de Faber-Krahn et conséquences, Actes de la Table Ronde de Géométrie Différentielle (Luminy, 1992), 205--232, Sémin. Congr., 1, Soc. Math. France, Paris, 1996.



\bibitem{cascatman} D. Castorina, G. Catino, C. Mantegazza, A triviality result for semilinear parabolic equations, Math. Eng. 4 (2022), no. 1, Paper No. 002, 15 pp.



\bibitem{catmanmaz} G. Catino, C. Mantegazza, L. Mazzieri, On the global structure of conformal gradient solitons with nonnegative Ricci tensor, Commun. Contemp. Math. 14 (2012), no. 6, 1250045, 12 pp.



\bibitem{checol} J. Cheeger, T. H. Colding, Lower bounds on Ricci curvature and the almost rigidity of warped products, Ann. of Math. (2) 144 (1996), no. 1, 189--237.

\bibitem{cheli} W. X. Chen, C. Li, Classification of solutions of some nonlinear elliptic equations, Duke Math. J. 63 (1991), no. 3, 615--622.

\bibitem{cheli2} W. X. Chen, C. Li, Qualitative properties of solutions to some nonlinear elliptic equations in $R^2$, Duke Math. J. 71 (1993), no. 2, 427--439.


\bibitem{cirfigron} G. Ciraolo, A. Figalli, A. Roncoroni, Symmetry results for critical anisotropic p-Laplacian equations in convex cones, Geom. Funct. Anal. 30 (2020), 770--803.

\bibitem{gidninir} B. Gidas, W. M. Ni, L. Nirenberg, Symmetry of positive solutions of nonlinear elliptic equations in $R^n$, Mathematical analysis and applications, Part A, pp. 369--402,
Adv. in Math. Suppl. Stud., 7a, Academic Press, New York-London, 1981.

\bibitem{gidspr} B. Gidas, J. Spruck, Global and local behavior of positive solutions of nonlinear elliptic equations, Comm. Pure Appl. Math. 34 (1981), no. 4, 525--598.


\bibitem{gri} A. Grigor'yan, Analytic and geometric background of recurrence and non-explosion of the Brownian motion on Riemannian manifolds, Bulletin of Amer. Math. Soc. 36 (1999) 135--249.


\bibitem{kim} S. Kim, Scalar curvature on noncompact complete Riemannian manifolds, Nonlinear Anal. 26:12 (1996), 1985--1993.



\bibitem{lizha} Y. Li, L. Zhang, Liouville-type theorems and Harnack-type inequalities for semilinear elliptic equations, J. Anal. Math. 90 (2003), 27--87.


\bibitem{mursoa} M. Muratori, N. Soave, Some rigidity results for Sobolev inequalities and related PDEs on Cartan-Hadamard manifolds, to appear in Ann. Sc. Norm. Super. Pisa Cl. Sci.


\bibitem{oba} M. Obata, The conjectures on conformal transformations of Riemannian manifolds,
J. Differential Geometry 6 (1971/72), 247--258.



\bibitem{tal} G. Talenti, Best constant in Sobolev inequality, Ann. Mat. Pura Appl. (4) 110 (1976), 353--372.


\bibitem{tas} Y. Tashiro, Complete Riemannian manifolds and some vector fields, Trans. Amer. Math. Soc. 117 (1965) 251--275.


\bibitem{wanzhu} F. Wang,X. Zhu, The structure of spaces with Bakry-\'Emery Ricci curvature bounded below, J. Reine Angew. Math. 757 (2019), 1--50.



\bibitem{wei} G. Wei, Yamabe equation on some complete noncompact manifolds,
Pacific J. Math. 302 (2019), no. 2, 717--739.





\end{thebibliography}

\

\

\

\end{document}